\documentclass[10pt,a4paper]{article}

\usepackage[final]{optional}

\usepackage{amsfonts, amsmath, wasysym}

\usepackage{tikz}

\usepackage{amssymb,amsthm,
paralist
}

\usepackage{
latexsym,
}

\usepackage{url}

\definecolor{darkgreen}{rgb}{0,0.5,0}
\definecolor{darkred}{rgb}{0.7,0,0}
\usepackage[colorlinks, 
citecolor=darkgreen, linkcolor=darkred
]{hyperref}

\theoremstyle{plain}
\newtheorem{lemma}{Lemma}[section]
\newtheorem{thm}[lemma]{Theorem}

\newtheorem{cor}[lemma]{Corollary}

\theoremstyle{definition}

\newtheorem{rmk}[lemma]{Remark}

\numberwithin{equation}{section}

\newcommand{\m}{\ensuremath{{\cal M}}}
\newcommand{\n}{\ensuremath{{\cal N}}}

\newcommand{\cb}{\ensuremath{{\cal B}}}

\newcommand{\calr}{\ensuremath{{\cal R}}}

\newcommand{\ti}{\tilde}

\newcommand{\al}{\alpha}
\newcommand{\be}{\beta}
\newcommand{\ga}{\gamma}

\newcommand{\de}{\delta}

\newcommand{\Om}{\Omega}

\newcommand{\si}{\sigma}

\newcommand{\vph}{\varphi}
\newcommand{\ep}{\varepsilon}

\newcommand{\R}{\ensuremath{{\mathbb R}}}
\newcommand{\N}{\ensuremath{{\mathbb N}}}

\newcommand{\downto}{\downarrow}

\newcommand{\grad}{\nabla}

\newcommand{\intersect}{\cap}

\DeclareMathOperator{\VolB}{VolB}

\DeclareMathOperator{\Interior}{Interior}
\DeclareMathOperator{\inj}{inj}

\def\blbox{\quad \vrule height7.5pt width4.17pt depth0pt}

\newcommand{\beq}{\begin{equation}}
\newcommand{\eeq}{\end{equation}}
\newcommand{\beqa}{\begin{equation}\begin{aligned}}
\newcommand{\eeqa}{\end{aligned}\end{equation}}
\newcommand{\brmk}{\begin{rmk}}
\newcommand{\ermk}{\end{rmk}}
\newcommand{\partref}[1]{\hbox{(\csname @roman\endcsname{\ref{#1}})}}
\newcommand{\half}{\frac{1}{2}}

\newcommand{\cmt}[1]{\opt{draft}{\textcolor[rgb]{0.5,0,0}{
$\LHD$ #1 $\RHD$\marginpar{\blbox}}}}

\newcommand{\bcmt}[1]{\opt{draft}{\textcolor[rgb]{0,0,0.9}{
$\LHD$ #1 $\RHD$\marginpar{\blbox}}}}

\newcommand{\Rm}{{\mathrm{Rm}}}
\newcommand{\Ric}{{\mathrm{Ric}}}

\usepackage{soul}

\title{{
\bf
Local mollification of Riemannian metrics using Ricci flow,
and Ricci limit spaces
} 
\\ 
\cmt{DRAFT with comments}
}
\author{Miles Simon and Peter M. Topping}
\date{6 May 2020}

\begin{document}

\maketitle
\parskip=10pt

\begin{abstract}
Given a three-dimensional Riemannian manifold containing a ball with an explicit lower bound on its Ricci curvature and positive lower bound on its volume, we use Ricci flow to perturb the Riemannian metric on the interior to a nearby Riemannian metric still with such lower bounds on its Ricci curvature and volume, but additionally with uniform bounds on its full curvature tensor and all its derivatives. The new manifold is near to the old one not just in the Gromov-Hausdorff sense, but also 
in the sense that the distance function is uniformly close to what it was before, and additionally  we have H\"older/Lipschitz equivalence of the old and new manifolds.

One consequence is that we obtain a local bi-H\"older correspondence between Ricci limit spaces in three dimensions and smooth manifolds.
This is more than a complete resolution
of the three-dimensional case of the conjecture of Anderson-Cheeger-Colding-Tian, describing how Ricci limit spaces in three dimensions must be homeomorphic to manifolds, and we obtain this in 
the most general, locally non-collapsed case.
The proofs build on results and ideas  from recent papers of Hochard and the current authors.
\end{abstract}

\parskip=-9pt
\tableofcontents
\parskip=10pt

\section{Introduction}

Since its introduction in 1982 by Hamilton, Ricci flow has become central to our efforts to understand the topology of manifolds that are either low dimensional, or that satisfy a curvature condition. See for example \cite{ham3D, P1, BW, BS}. Loosely speaking, the strategy is to allow Ricci flow to evolve such a manifold into one that we can recognise.
In this paper we investigate the use of Ricci flow in a different direction. 
We study its use  as a local mollifier of Riemannian metrics,
taking an initial metric that is coarsely controlled on some local region,
and giving back a 
smooth metric satisfying good estimates on a slightly smaller region.

We will apply our theory in order to understand Ricci limit spaces, i.e. Gromov-Hausdorff limits of sequences of manifolds with a uniform lower bound on their Ricci tensors, and a uniform positive lower bound on the volume of one unit ball.

Our main Ricci flow result is the following theorem, which applies to Riemannian 3-manifolds that are not assumed to be complete.
The result without the uniform conclusions on Ricci curvature, volume 
and distance follows from recent work of Hochard \cite{hochard}, and our proof crucially uses some of his ideas as we describe later.
We will use our uniform conclusions on Ricci curvature, volume 
and distance in applications to Ricci limit spaces.

\bcmt{adjusting the theorem below to make $c_0$ independent of $\al_0$ would require making the hypotheses messier. (We'd have to restate the volume lower bound hypothesis to apply to many balls.) I think it's easier to leave like this because anyone can come after and apply our estimates to blow-ups  to give the better dependency. Could think about the dependencies on $\ep$. There is no dependency of $v$ on $\ep$, but this is so obvious, it is barely worth the extra ink to spell it out.}

\newcommand{\wasB}{{\mathfrak{B}}}

\begin{thm}[Mollification theorem]
\label{mollifier_thm}
Suppose that $(M^3,g_0)$ is a Riemannian manifold, $x_0\in M$,  $B_{g_0}(x_0,1)\subset\subset M$, with geometry controlled by
\begin{equation}
\left\{
\begin{aligned}
& \Ric_{g_0}\geq -\al_0<0 \qquad & &\text{on } B_{g_0}(x_0,1)\\
& \VolB_{g_0}(x_0,1)\geq v_0>0 
\end{aligned}
\right.
\end{equation}
Then for any $\ep\in (0,1/10)$, there exist $T,v,\al,c_0>0$ depending only on $\al_0$, $v_0$ and $\ep$,
and a Ricci flow $g(t)$ defined for $t\in [0,T]$ on 
the slightly smaller ball $\wasB:=B_{g_0}(x_0,1-\ep)$, with
$g(0)=g_0$ on $\wasB$, such that
for each $t\in [0,T]$, $B_{g(t)}(x_0,1-2\ep)\subset\subset\wasB$ where the Ricci flow is defined,  and 
\begin{equation}
\label{moll_thm_conc1}
\left\{
\begin{aligned}
& \Ric_{g(t)}\geq -\al \qquad & &\text{on } 
\wasB\\
& \VolB_{g(t)}(x_0,1-2\ep)\geq v>0. & &\\
\end{aligned}
\right.
\end{equation}
Moreover, the curvature tensor and all its derivatives are controlled according to
\begin{equation}
\left\{
\begin{aligned}
\label{curv_control_moll}
& |\Rm|_{g(t)}\leq c_0/t \\
& \left|\grad^k \Rm\right|_{g(t)}\leq 
C(k,\al_0,v_0,\ep)/t^{1+\frac{k}{2}}
\end{aligned}
\right.
\end{equation}
on $\wasB$ for all $t\in (0,T]$ and $k\in\N$.
If we fix $s\in [0,T]$ and $x,y\in B_{g(s)}(x_0,\half-2\ep)$, then
$x,y\in B_{g(t)}(x_0,\half-\ep)$ for all $t\in [0,T]$ so 
the infimum length of curves connecting $x$ and $y$ within $(\wasB,g(t))$ is realised by a geodesic within $B_{g(t)}(x_0,1-2\ep)\subset\wasB$ 
where the Ricci flow is defined.
Moreover, for any $t\in [0,T]$, and such $x$ and $y$,
we have
\beq
\label{main_dist_est3}
d_{g_0}(x,y)-\be\sqrt{c_0 t}
\leq d_{g(t)}(x,y)
\leq e^{\al t}d_{g_0}(x,y),
\eeq
where $\be=\be(n)\geq 1$, and the H\"older estimate
$$d_{g_0}(x,y)\leq \ga\left[d_{g(t)}(x,y)\right]^\frac{1}{1+4c_0}$$
where $\ga<\infty$ depends only on $c_0$.
\end{thm}

\brmk
We clarify that within a Riemannian manifold $(M,g)$, containing a point $x$, we write $\VolB_g(x,r)$ for the volume (with respect to $g$) of a ball $B_g(x,r)$ centred at $x$ of radius $r$ (with respect to $g$). Later we allow $r\in \R$, and if $r\leq 0$, the ball is empty and the volume zero.
\ermk


There is an implicit subtlety in considering Riemannian distances on incomplete Riemannian manifolds that we avoid in the theorem above through the assertions that 
$$\textstyle
x,y\in B_{g(t)}(x_0,\half-\ep)\quad\text{and}\quad
B_{g(t)}(x_0,1-2\ep)\subset\subset\wasB,$$
as we explain in the following general remark (with $R=1-2\ep$).

\brmk
\label{dist_rmk}
If we have a Riemannian manifold $(N,g)$, and we are considering a ball $B_g(x_0,R)\subset\subset N$, then we can make sense of the distance between two points $x,y\in B_g(x_0,R)$ in two ways -- either as the infimum length of connecting paths that remain within $B_g(x_0,R)$, or as the infimum length of such paths that may stray anywhere within $N$. In general, the latter distance will be shorter. However, we observe that if we are sure that $x$ and $y$ lie within the \emph{half} ball
$B_g(x_0,R/2)$, then these two distances must agree,
and there exists a minimising geodesic between $x$ and $y$ that remains
within $B_g(x_0,R)$.
\ermk

Locally defined Ricci flows were  considered by D.Yang, assuming volume bounds from below, supercritical integral bounds on the initial Ricci curvature and a smallness condition on
the  $L^{n/2}$ norm of the initial Riemannian curvature tensor. See \cite[Theorem 9.2]{DY}. 
One could also draw a comparison with the theory of Ricci flows on manifolds with boundary, see the work of Gianniotis \cite{panagiotis} and the references therein.
They are also considered in the work of Hochard \cite{hochard}.

In this paper we apply the local Ricci flow we construct in Theorem \ref{mollifier_thm}
to understand the geometry of so-called Ricci limit spaces, i.e. the metric spaces that arise as limits of manifolds with a uniform lower Ricci bound, and a positive uniform lower bound on the volume of \emph{one} unit ball. 
Such limits need not be smooth; it is easy to see that a metric cone can arise as such a limit, even in two dimensions, and similarly we can even have a dense set of cone points
(e.g. \cite[Example 0.29]{cheeger_book}). All such cone points are singular points, defined to be the points in the limit for which there exists a tangent cone that is not Euclidean space.
Nevertheless, the following theorem establishes that in (up to) three dimensions, Ricci limit spaces are locally bi-H\"older to smooth manifolds.

\bcmt{the compactness of $(\cb,d_0)$ below is telling us that it does not hang off the edge of $\m$. And there are no punctures in $\m$ etc.}

\begin{thm}[Ricci limits in 3D are bi-H\"older homeomorphic to smooth manifolds locally]
\label{preACCT}
Suppose that $(M^3_i,g_i)$ is a sequence of (not necessarily complete) Riemannian manifolds such that for all $i$ we have $x_i\in M_i$,  $B_{g_i}(x_i,1)\subset\subset M_i$, and 
\begin{equation}
\left\{
\begin{aligned}
& \Ric_{g_i}\geq -\al_0<0 \qquad & &\text{on } B_{g_i}(x_i,1)\\
& \VolB_{g_i}(x_i,1)\geq v_0>0. 
\end{aligned}
\right.
\end{equation}
Then there exist a smooth three-dimensional manifold without boundary $\m$, containing a point $x_0$,  and a metric\footnote{in the sense of metric spaces}
$d_0:\m\times \m\to [0,\infty)$ generating the same topology as $\m$
such that $B_{d_0}(x_0,1/10)\subset\subset\m$ and so that
after passing to a subsequence the compact metric spaces 
$(\overline{B_{g_i}(x_i,1/10)},d_{g_i})$ Gromov-Hausdorff converge to 
$(\cb,d_0)$, where $\cb=\overline{B_{d_0}(x_0,1/10)}$.

Moreover there exists a sequence of smooth maps 
$\vph_i:\m\to B_{g_i}(x_i,1/2)\subset M_i$, mapping $x_0$ to $x_i$,
diffeomorphic onto their images, such that 
$$d_{g_i}(\vph_i(x),\vph_i(y))\to d_0(x,y)
\qquad\text{uniformly for }(x,y)\in \cb\times\cb,$$
as $i\to\infty$, and for all $\eta\in (0,1/100)$, 
\beq
\label{sandwich2}
B_{g_i}(x_i,1/10-\eta)\subset
\vph_i(\cb)\subset
B_{g_i}(x_i,1/10+\eta),
\eeq
for sufficiently large $i$.

Finally, for any \emph{smooth}  Riemannian metric $g$ on $\m$, the identity map 
$(\cb,d_g)\to (\cb,d_0)$ is H\"older continuous.
Conversely, the identity map 
$(\cb,d_0)\to (\cb,d_g)$ is Lipschitz continuous.
\end{thm}

\bcmt{we talk about GH convergence rather than pGH convergence above, but the pGH convergence follows from the other parts of the theorem. We need pGH convergence in the proof of Cor \ref{ACCTcomplete}.}

\bcmt{no longer ask that $g$ is complete above}

Here, and elsewhere in this paper, we assume implicitly that all manifolds are connected.

The assertion concerning the maps $\vph_i$ is not only implying the 
previous Gromov-Hausdorff convergence, even pointed  Gromov-Hausdorff convergence,
but is also
saying that we may essentially take the Gromov-Hausdorff approximations to be smooth.

\bcmt{I put the images of $\vph_i$ as $B_{g_i}(x_i,1/2)$ to be sure that the distance function in the image is all OK, though we could make it a lot smaller than that}

To clarify, when we write $(\overline{B_{g_i}(x_i,1/10)},d_{g_i})$, we are referring to the metric $d_{g_i}$ defined on the whole of
$(M_i,g_i)$, which is subsequently restricted to $\overline{B_{g_i}(x_i,1/10)}$, which 
need not agree with the infimum of lengths of connecting paths that remain within this $1/10$-ball itself, cf. Remark \ref{dist_rmk}.

The proof of this theorem, given in Section \ref{ricci_lims_sect}, relies on being able to flow the sequence $(M_i,g_i)$ locally using the Mollification theorem \ref{mollifier_thm}. In particular, the lower Ricci bounds of that theorem will be crucial, and these in turn rely on the work in \cite{ST1}.

\bcmt{There is no claim that $(\m,d_0)$ is a length space.
It should be something like a `local' length space - i.e. points in a smaller subset can be connected by a minimising geodesic, but that might go outside the smaller subset..
Indeed, the flow could be basically a flat cylinder, and $\m$ could correspond to a ball that almost wraps all the way round the cylinder..
In the corollary below the limit is a length space, but this is nothing to do with us so I leave it implicit.
}

As indicated earlier, the regular set of a Ricci limit space is the set of points at which all tangent cones are Euclidean space (here $\R^3$). There is a bigger space, the $\ep$-regular set
$\calr_\ep$ that is roughly the set of points at which the limit is within $\ep$ of being Euclidean on sufficiently small scales, in a scaled sense. In the example above with a possibly dense set of cone points, $\calr_\ep$
would include the cone points  that have sufficiently small cone angle. It is a theorem of Cheeger-Colding \cite[Theorem 5.14]{CC1} that for sufficiently small $\ep>0$, the interior of $\calr_\ep$ is bi-H\"older homeomorphic to a smooth manifold. In particular, in the  special case that 
the singular set is empty, their theorem implies bi-H\"older equivalence.

Theorem \ref{preACCT} has various 
corollaries, for example:


\cmt{changed points below to $y$ from $x$. Hopefully changed all instances, but needs checking}

\begin{cor}[Conjecture of Anderson, Cheeger, Colding, Tian, 3D case]
\label{ACCTcomplete}
Suppose that $(M^3_i,g_i)$ is a sequence of complete Riemannian manifolds such that for all $i$ we have $y_i\in M_i$,  and 
\begin{equation}
\label{ACCTcomplete_ests}
\left\{
\begin{aligned}
& \Ric_{g_i}\geq -\al_0<0 \qquad & &\text{throughout } M_i\\
& \VolB_{g_i}(y_i,1)\geq v_0>0. 
\end{aligned}
\right.
\end{equation}
Then there exist a   three-dimensional topological manifold $M$ and a   metric $d:M\times M\to [0,\infty)$ generating the same topology as $M$ and making $(M,d)$ a complete metric space,
such that 
after passing to a subsequence, we have
$$(M_i,d_{g_i},y_i)\to (M,d,y_0)$$
in the pointed Gromov-Hausdorff sense, for some $y_0\in M$.
The charts of $M$  can be taken to be bi-H\"older with respect to $d$.
\end{cor}

For the definition of pointed Gromov-Hausdorff convergence, see
\cite[Definition 8.1.1]{BBI}.

\cmt{although they have slightly different definition for metric space, allowing infinite distances - check no issues created by this.
}

\cmt{Deleted the following comment, see also BBI Ex 8.1.4.\\
The pointed Gromov-Hausdorff convergence assertion here is equivalent to the statement that for any $r>0$, the closed balls
$\overline{B_{g_i}(y_i,r)}$ equipped with the distance metric $d_{g_i}$
Gromov-Hausdorff converge to the limit 
$\overline{B_d(y_0,r)}$. 
}

Corollary \ref{ACCTcomplete} solves the conjecture made in \cite{CC1}, Conjecture  0.7 for three-dimensional manifolds. See also, for example \cite[Remark 1.19]{CCT}
and \cite[Remark 10.23]{cheeger_book}.
A related conjecture of Anderson, see \cite[Conjecture 2.3]{AndC} and
\cite[Conjecture  0.8]{CC1}, constrains the singular set when one assumes in addition that the Ricci curvature is bounded from \emph{above}. In three dimensions, that conjecture would reduce to the singular set being empty, which would imply the bi-H\"older equivalence that we have proved.
Another situation in which the singular set can be shown to be empty in three-dimensions is when one assumes, in addition, $L^p$ control on the curvature tensor, for appropriately large $p$
\cite{CCT}.
In contrast, our work is allowing cone points in the limit, and indeed cone points that are not close to being Euclidean.

\bcmt{I leave implicit that the limit is a length space}


Of course, the existence of some Gromov-Hausdorff limit, after passing to a subsequence, is an immediate consequence of Gromov compactness.
Our results says that this limit
can be assumed to be homeomorphic to a smooth manifold. This is in contrast to the higher dimensional case, as can be seen by blowing down the Eguchi-Hanson (Ricci flat) metric, which will converge to the quotient of $\R^4$ by the map $x\mapsto -x$, which is \emph{not} homeomorphic to any manifold. The conjecture for higher dimensions is that 
the interior of the complement of the codimension 4 stratum of the singular set should be homeomorphic to a topological manifold \cite[Conjecture  0.7]{CC1}.

Note that the noncollapsedness condition in the corollary above is that the volume of \emph{one} unit ball is uniformly bounded below. 
It is  expected that it is \emph{not} possible to 
start the Ricci flow globally with such a manifold. If we weaken the claim of the corollary and assume a uniform lower bound on the volume of \emph{all} unit balls, i.e. even as we move the centre-point around, then one now \emph{can} start the Ricci flow globally (cf. Theorem \ref{global_exist_thm} or \cite{hochard}) and the result would follow from a combination of \cite{ST1} and \cite{hochard}. Raphael Hochard has informed us that a future revision of his paper \cite{hochard} will contain a complete proof of this latter result. His current preprint proceeds without establishing uniform lower Ricci bounds.
Earlier results of the first author have established this result under additional hypotheses
\cite{MilesCrelle3D}.

We will give the proof of Corollary \ref{ACCTcomplete} in Section \ref{ACCTcomplete_sect}.
The same proof will work if we merely assume that for each $r>0$, the Ricci curvature is uniformly bounded below on balls $B_{g_i}(y_i,r)$, with the bound depending on $r$ but being independent of $i$.

\bcmt{If we wanted to get rid of the $\al_0$ dependency for $c_0$ below, then we'd probably need to change the volume hypothesis to give a lower bound for the volume of balls of arbitrarily small radius. We get such a lower bound below, but it depends on $\al_0$.}

The existence part of the Mollification theorem \ref{mollifier_thm}
will be deduced by rescaling the following 
main local (in space) existence theorem. 
Some important ideas in the proof arise first in the work of Hochard 
\cite{hochard}, where related results without the uniform Ricci control were proved.

\begin{thm}[Local existence theorem in 3D]
\label{main_exist_thm2}
Suppose $s_0\geq4$.
Suppose further that $(M^3,g_0)$ is a Riemannian manifold, $x_0\in M$,  $B_{g_0}(x_0,s_0)\subset\subset M$, and 
\begin{equation}
\left\{
\begin{aligned}
& \Ric_{g_0}\geq -\al_0<0 \qquad & &\text{on } B_{g_0}(x_0,s_0)\\
& \VolB_{g_0}(x,1)\geq v_0>0 \qquad & &\text{for all }x\in B_{g_0}(x_0,s_0-1).
\end{aligned}
\right.
\end{equation}
Then there exist $T,\al,c_0>0$ depending only on $\al_0$ and $v_0$,
and a Ricci flow $g(t)$ defined for $t\in [0,T]$ on $B_{g_0}(x_0,s_0-2)$, with
$g(0)=g_0$ where defined, such that
\begin{equation}
\left\{
\begin{aligned}
& \Ric_{g(t)}\geq -\al \qquad & &\text{on } B_{g_0}(x_0,s_0-2)\\
& |\Rm|_{g(t)}\leq c_0/t \qquad & &\text{on } B_{g_0}(x_0,s_0-2)
\end{aligned}
\right.
\end{equation}
for all $t\in (0,T]$.
\end{thm}

The curvature upper and lower bounds in this theorem will give us good control on the distance function, via  a  refinement of \cite[Lemma 3.4]{ST1}, valid in any dimension, that is given in
Lemma \ref{biholder_dist2} below.
In particular, although $g(t)$ is not assumed to be complete, this control will guarantee that the Ricci flow is `big enough'.

Theorem \ref{main_exist_thm2} is stated in a way that allows us to apply it on larger and larger balls within a \emph{complete} Riemannian manifold $(M^3,g_0)$, to give local Ricci flows with enough regularity to give compactness. That is, a subsequence will converge to a Ricci flow starting with $(M^3,g_0)$ that enjoys the same geometric control as the initial data. The existence of a Ricci flow without the uniform control on Ricci curvature and volume, and without the control on how the distance function can increase,  was first proved by Hochard \cite{hochard}, and the following result could be proved indirectly by combining the results of
\cite{hochard} and \cite{ST1}.

\begin{thm}[Global existence theorem in 3D]
\label{global_exist_thm}
Suppose that $(M^3,g_0)$ is a complete Riemannian manifold with the properties that
\begin{equation}
\left\{
\begin{aligned}
& \Ric_{g_0}\geq -\al_0<0 & &\\
& \VolB_{g_0}(x,1)\geq v_0>0 \qquad & &\text{for all }x\in M.
\end{aligned}
\right.
\end{equation}
Then there exist $T,v,\al,c_0>0$ depending only on $\al_0$ and $v_0$,
and a smooth, complete Ricci flow $g(t)$ defined for $t\in [0,T]$ on $M$, with
$g(0)=g_0$, such that
\begin{equation}
\label{global_conc_ests}
\left\{
\begin{aligned}
& \Ric_{g(t)}\geq -\al & &\\
& \VolB_{g(t)}(x,1)\geq v>0 \qquad & &\text{for all }x\in M\\
& |\Rm|_{g(t)}\leq c_0/t \qquad & &\text{throughout } M
\end{aligned}
\right.
\end{equation}
for all $t\in (0,T]$. Moreover, 
for any $0\leq t_1\leq t_2\leq T$, and any $x,y\in M$, we have
\beq
\label{global_dist_est}
d_{g(t_1)}(x,y)-\be\sqrt{c_0}(\sqrt{t_2}-\sqrt{t_1})
\leq d_{g(t_2)}(x,y)
\leq e^{\al(t_2-t_1)}d_{g(t_1)}(x,y),
\eeq
where $\be=\be(n)$.
\end{thm}

We prove Theorem \ref{global_exist_thm} in Section \ref{global_sect}.

Finally, as a variation of the ideas above, we record that it is possible to run the Ricci flow starting with appropriate rough data.

\begin{thm}[Ricci flow from a uniformly noncollapsed 3D Ricci limit space]
\label{RFfromRLS}
Suppose that $(M^3_i,g_i)$ is a sequence of complete Riemannian manifolds such that for all $i$ we have $x_i\in M_i$,  and 
\begin{equation}
\left\{
\begin{aligned}
& \Ric_{g_i}\geq -\al_0<0 \qquad & &\text{throughout } M_i\\
& \VolB_{g_i}(x,1)\geq v_0>0 \qquad & &\text{ for all }x\in M_i. 
\end{aligned}
\right.
\end{equation}
Then there exists a smooth manifold $M$, a point $x_\infty\in M$, a complete Ricci flow $g(t)$ on $M$ for $t\in (0,T]$, where $T>0$ depends only on $\al_0$ and $v_0$, 
and a continuous  distance metric $d_0$ on $M$ such that
$d_{g(t)}\to d_0$ locally uniformly as $t\downto 0$, 
and after 
passing to a subsequence in $i$ we have that
$(M_i,d_{g_i},x_i)$ converges in the pointed Gromov-Hausdorff sense to 
$(M,d_0,x_\infty)$.
Furthermore, the Ricci flow satisfies the estimates
\eqref{global_conc_ests} and \eqref{global_dist_est}
for some $\al,v,c_0>0$ depending only on $\al_0$ and $v_0$.
\end{thm}

\bcmt{currently not including the biHolder distance estimates in these final theorems. You can always get them with hindsight}

We give the proof of Theorem \ref{RFfromRLS}
in Section \ref{RF_from_RLS_sect}.

\emph{Recent progress:}
Since this paper was released on arXiv, there have been several refinements and applications of our work made.
In \cite{LeeTam}, Lee and Tam use the ideas in this work to address a special case of Yau's uniformisation conjecture. 
In \cite{MT1}, our work is refined to give a global version of Corollary \ref{ACCTcomplete} using pyramid Ricci flows.
In \cite{lai}, Lai demonstrates that our proof works in higher dimensions if an appropriate substitution for the Double Bootstrap Lemma \ref{modifieddouble2} is made, based on the estimates of Bamler, Cabezas-Rivas and Wilking \cite{B_CR_W}. These latter estimates have since been localised by Hochard \cite{hochard_thesis}, giving a higher-dimensional version of our Double Bootstrap Lemma
\ref{modifieddouble2} for suitably stronger curvature hypotheses. This subsequently allows a global higher-dimensional result to be proved \cite{MT2}. 
Lai's work \cite{lai} also incorporates the bi-H\"older estimates developed in Lemma 3.1 of this paper.
Finally, it has been pointed out to us by De Philippis, Mondino and Gigli separately that our work settles the question of whether 
noncollapsed Ricci limit spaces 
coincide with the corresponding RCD spaces -- see \cite[\S 3.3, Remark 4]{cetraro} for a discussion.

\emph{Acknowledgements:} 
The first author was supported by the Priority Program `Geometry at Infinity' (SPP 2026) of the German Research Foundation (DFG).
The second author was 
supported by EPSRC grant number EP/K00865X/1.


\section{Expanding and Shrinking Balls lemmata}

\bcmt{could add a `H\"older' balls lemma, but we don't need it}

In this section we record the two main local ball inclusion lemmas from \cite{ST1}.
In fact, the first is a refinement of the  result from \cite{ST1}, which we will need in 
Section \ref{ricci_lims_sect}.

\begin{lemma}[{The expanding balls lemma, cf. \cite[Lemma 3.1]{ST1}}]
\label{EBL2}
Suppose $(M,g(t))$ is a Ricci flow for $t\in [-T,0]$, $T>0$, on a manifold $M$ of any dimension.
Suppose that $x_0\in M$ and that $B_{g(0)}(x_0,R)\subset\subset M$
and $\Ric_{g(t)}\geq -K<0$ on 
$B_{g(0)}(x_0,R)\intersect B_{g(t)}(x_0,Re^{Kt})\subset B_{g(t)}(x_0,R)$
for each $t\in [-T,0]$.
Then
\beq
\label{claimlinenewEBL}
B_{g(0)}(x_0,R)\supset B_{g(t)}(x_0,Re^{Kt})
\eeq
for all $t\in [-T,0]$.
\end{lemma}

\begin{proof}
It suffices to show that for $\al\in (0,1)$ arbitrarily close to $1$
we have
\beq
\label{alpha_claim}
B_{g(0)}(x_0,R)\supset \overline{B_{g(t)}(x_0,R\al e^{\frac{Kt}{\al}})}
\text{ for all }t\in [-T,0]. 
\eeq
This assertion is clearly true for $t=0$, and by smoothness, also for $t<0$ sufficiently close to $0$ (depending on $\al$ amongst other things).
If \eqref{alpha_claim} were not true, then we could let $t_0\in [-T,0]$ be the supremum of the times at which the inclusion fails, and by smoothness it will fail also at time $t_0$.
Thus we can find $y\in \overline{B_{g(t_0)}(x_0,R\al e^{\frac{Kt_0}{\al}})}$ such that $d_{g(0)}(x_0,y)=R$.
Pick a minimising unit speed geodesic $\ga$, with respect to $g(t_0)$, connecting $x_0$ and $y$ that lies within both
$\overline{B_{g(t_0)}(x_0,R\al e^{\frac{Kt_0}{\al}})}
\subset B_{g(t_0)}(x_0,R e^{Kt_0})$
and 
$\overline{B_{g(0)}(x_0,R)}$
where $\Ric_{g(t_0)}\geq -K$.
We have 
\beq
\label{t0length}
L_{g(t_0)}(\ga)\leq R\al e^{\frac{Kt_0}{\al}},
\eeq
and by the flow equation
\beq
\label{length_deriv}
\frac{d}{dt}\bigg|_{t=t_0}L_{g(t)}(\ga)
=-\int_\ga \Ric(\dot\ga,\dot\ga)\leq K L_{g(t_0)}(\ga).
\eeq
Because $t_0$ is the supremum time at which the inclusion of
\eqref{alpha_claim} fails, for $t\in (t_0,0]$ we have
$L_{g(t)}(\ga)> R\al e^{\frac{Kt}{\al}}$, and subtracting \eqref{t0length} gives
$$L_{g(t)}(\ga)-L_{g(t_0)}(\ga)
> R\al (e^{\frac{Kt}{\al}}-e^{\frac{Kt_0}{\al}}).$$
Dividing by $t-t_0>0$ and taking a limit $t\downto t_0$ gives
$$\frac{d}{dt}\bigg|_{t=t_0}L_{g(t)}(\ga)
\geq R Ke^{\frac{Kt_0}{\al}}\geq \frac{1}{\al}K L_{g(t_0)}(\ga),$$
which contradicts \eqref{length_deriv}.
\end{proof}

\begin{lemma}[{The shrinking balls lemma, \cite[Corollary 3.3]{ST1}}]
\label{SBL}
Suppose $(M,g(t))$ is a Ricci flow for $t\in [0,T]$ on a manifold $M$ of any dimension $n$.
Then there exists $\be=\be(n)\geq 1$ such that the following is true.
Suppose $x_0\in M$ and that $B_{g(0)}(x_0,r)\subset\subset M$ 
for some $r>0$,  
and $|\Rm|_{g(t)}\leq c_0/t$,
or more generally $\Ric_{g(t)}\leq (n-1)c_0/t$,
on $B_{g(0)}(x_0,r)\intersect B_{g(t)}(x_0,r-\be\sqrt{c_0 t})$ 
for each $t\in (0,T]$ and some $c_0>0$.
Then 
\beq
B_{g(0)}(x_0,r)\supset B_{g(t)}\left(\textstyle{x_0,r-\be\sqrt{c_0 t}}\right)
\eeq
for all $t\in [0,T]$.
More generally, for $0\leq s\leq t\leq T$, we have 
$$B_{g(s)}\left(\textstyle{x_0,r-\be\sqrt{c_0 s}}\right)
\supset B_{g(t)}\left(\textstyle{x_0,r-\be\sqrt{c_0 t}}\right).$$
\end{lemma}

\bcmt{maybe worth noting here that the shrinking balls corollary can be strengthened a bit trivially by only requiring hypotheses between the times considered, not for times all the way down to $t=0$}

\section{Distance estimates and bi-H\"older equivalence}

In this section we explain how the local control on the curvature that we typically have in this paper leads to 
control on the distance function. The main novelty is a H\"older estimate \eqref{holder_est}.


\bcmt{beware: $\Om_T$ as defined below could be disconnected}

\begin{lemma}[Bi-H\"older distance estimate]
\label{biholder_dist2}
Suppose $(M^n,g(t))$ is a  Ricci flow for $t\in (0,T]$, not necessarily complete, with the property that for some $x_0\in M$ and $r>0$, and all
$t\in (0,T]$, we have $B_{g(t)}(x_0,2r)\subset\subset M$.
Suppose further that for some $c_0, \al>0$, and for each
$t\in (0,T]$,  we have
\begin{equation}
\left\{
\begin{aligned}
& \Ric_{g(t)}\geq -\al 
\\
& \Ric_{g(t)}\leq \frac{(n-1)c_0}{t}
\end{aligned}
\right.
\end{equation}
throughout $B_{g(t)}(x_0,2r)$.
By Remark \ref{dist_rmk}, for all $x,y\in \Om_T\subset M$, where
$$\Om_T:=\bigcap_{t\in (0,T]}B_{g(t)}(x_0,r),$$
the distance $d_{g(t)}(x,y)$ is unambiguous for all $t\in (0,T]$ and must be realised by a minimising geodesic lying within $B_{g(t)}(x_0,2r)$. Then 
for any 
$0< t_1\leq t_2\leq T$, we have
\beq
\label{main_dist_est2}
d_{g(t_1)}(x,y)-\be\sqrt{c_0}(\sqrt{t_2}-\sqrt{t_1})
\leq d_{g(t_2)}(x,y)
\leq e^{\al(t_2-t_1)}d_{g(t_1)}(x,y),
\eeq
where $\be=\be(n)$.
In particular, the distance metrics $d_{g(t)}$ converge uniformly to a distance metric $d_0$ on $\Om_T$ as $t\downto 0$, and 
\beq
\label{specialised_main_dist_est2}
d_0(x,y)-\be\sqrt{c_0t}
\leq d_{g(t)}(x,y)
\leq e^{\al t}d_0(x,y),
\eeq
for all $t\in (0,T]$.
Furthermore, there exists $\ga>0$ depending on $n$, $c_0$ and upper bounds for $r$ and $T$
such that
\beq
\label{holder_est}
d_{g(t)}(x,y)\geq \ga \left[d_0(x,y)\right]^{1+2(n-1)c_0},
\eeq
for all $t\in (0,T]$.
Finally, for all $t\in (0,T]$ and $R<R_0:=re^{-\al T}-\be\sqrt{c_0 T}<r$, we have
\beq
\label{ball_inclusions}
B_{g(t)}(x_0,R_0)\subset\m\subset \Om_T\qquad\text{and}\qquad
B_{d_0}(x_0,R)\subset\subset \m,
\eeq
where $\m$ is the component of 
$\Interior(\Om_T)$ containing $x_0$.
\end{lemma}

\bcmt{care was required above in \eqref{ball_inclusions}. The first inclusion cleanly implies that $B_{g(t)}(x_0,R)\subset\subset \Interior(\Om_T)$. But the same would not work for the second inclusion - note that the metric $d_0$ is only even defined on $\Om_T$!}

To clarify, the first inclusion of \eqref{ball_inclusions} will be vacuous if $T$ is sufficiently large that $R_0$ is nonpositive.

\bcmt{we're leaving implicit that $d_0$ induces the same topology as we already have on $M$. This is following from 
\eqref{specialised_main_dist_est2} and \eqref{holder_est} because $d_{g(t)}$ must generate $M$'s topology}

\bcmt{$\Om_T$ is nonempty since it includes $x_0$, but we need these final quantitative statements.
Note that $\Om_T$ can be closed. It might have been \emph{just}
$x_0$ without the final assertion about time $t$ balls!}

\brmk
A consequence of this lemma is that 
the identity map $(\m,d_{g(t)})\to (\m,d_0)$ is H\"older continuous with H\"older exponent $[1+2(n-1)c_0]^{-1}$,
and the
identity map $(\m,d_0)\to (\m,d_{g(t)})$ in the other direction is Lipschitz.
The H\"older exponent can be seen to be sharp on  solitons coming out of cones in two dimensions.
\ermk

\brmk
\label{got0}
We have stated the lemma for Ricci flows that exist for $t\in (0,T]$, since that is the situation for limit flows considered in this paper. However, we can apply the lemma to Ricci flows that exist for $t\in [0,T]$, in which case we automatically have $d_0=d_{g(0)}$, and the conclusions of the lemma hold on the whole interval $[0,T]$.
\ermk

\brmk
Estimate \eqref{holder_est} only uses the \emph{upper} curvature bound of the lemma, i.e. the Ricci lower bound is not required.
In contrast to earlier distance estimates, we prove it by splitting the time interval $[0,t]$ into two subintervals $[0,t_3]$
and $[t_3,t]$, and controlling the distance on each using 
different techniques.
\ermk

\begin{proof}[Proof of Lemma \ref{biholder_dist2}]
The function $t\mapsto d_{g(t)}(x,y)$ is 
locally Lipschitz on $(0,T]$.
At times at which this function is differentiable, we have
$$\frac{d}{dt}d_{g(t)}(x,y)=\frac{d}{dt}L_{g(t)}(\ga),$$
where $\ga$ is a unit-speed minimising geodesic from $x$ to $y$ at the given time.
On the one hand, we have
\begin{align}
\label{ddtd_formula}
\frac{d}{dt}d_{g(t)}(x,y) &=\frac{d}{dt}L_{g(t)}(\ga)=-\int_\ga\Ric(\dot\ga,\dot\ga)\\
&\leq \al\, d_{g(t)}(x,y),
\end{align}
and integrating in time gives the second inequality of 
\eqref{main_dist_est2}.
On the other hand, we know from Hamilton-Perelman \cite[(3.6)]{ST1}
that
$$\frac{d}{dt}d_{g(t)}(x,y)=\frac{d}{dt}L_{g(t)}(\ga)\geq -\frac{\be}{2}\sqrt{c_0/t},$$
where $\be=\be(n)\geq 1$ is ultimately the $\be$ that appears in Lemma \ref{SBL}.
Integrating in time now gives the first inequality of 
\eqref{main_dist_est2}.
Having established \eqref{main_dist_est2}, the existence of $d_0$, and the estimate
\eqref{specialised_main_dist_est2}, is obvious.
Note that the second inequality of \eqref{specialised_main_dist_est2} guarantees that $d_0(x,y)=0$ implies $x=y$, as required for $d_0$ to be a metric.

The first step to proving \eqref{holder_est} is to observe that for
$$0<t\leq t_3:=\frac{1}{c_0}\left[\frac{1}{2\be}d_0(x,y)\right]^2,$$
the first inequality of \eqref{specialised_main_dist_est2} implies that
\beq
\label{first_time_interval2}
d_{g(t)}(x,y)\geq \half d_0(x,y).
\eeq
In particular, for $t\leq t_3$ we have established 
a stronger conclusion than our desired  \eqref{holder_est}.
On the other hand, to prove \eqref{holder_est} for 
times after $t_3$, we first use \eqref{first_time_interval2} at the last possible time $t_3$, giving
\beq
\label{first_time_interval_t_3}
d_{g(t_3)}(x,y)\geq \half d_0(x,y),
\eeq
and then 
use the consequence of \eqref{ddtd_formula} that
$$\frac{d}{dt} d_{g(t)}(x,y)\geq -(n-1)\frac{c_0}{t} d_{g(t)}(x,y),$$
at times at which $t\mapsto d_{g(t)}(x,y)$ is differentiable,
which implies, when integrated from time $t_3$ to any later time $t>t_3$, that 
$$d_{g(t)}(x,y)\geq d_{g(t_3)}(x,y)\left[
\frac{t}{t_3}\right]^{-(n-1)c_0}.$$
Inserting the formula for $t_3$, and using \eqref{first_time_interval_t_3} we obtain
\beq
\label{lower_dist_bd_holder2}
d_{g(t)}(x,y)\geq 
\ga(n,c_0)
\left[d_0(x,y)\right]^{1+2(n-1)c_0}
t^{-(n-1)c_0},
\eeq
which is a little stronger than required.

It remains to prove the inclusions \eqref{ball_inclusions}.
On the one hand, for $s\in (0,t)$ we have
$R_0+\be\sqrt{c_0(t-s)} <R_0+\be\sqrt{c_0 T}
=re^{-\al T}<r$, and so by Lemma \ref{SBL} applied from time $s$ onwards, we have
\beqa
B_{g(t)}(x_0,R_0) &\subset B_{g(s)}(x_0,R_0+\be\sqrt{c_0(t-s)})\\
&\subset B_{g(s)}(x_0,R_0+\be\sqrt{c_0 T})=B_{g(s)}(x_0,re^{-\al T})\\
&\subset B_{g(s)}(x_0,r).
\eeqa
On the other hand, for $s\in (t,T]$, we have
$R_0e^{\al (s-t)}< R_0e^{\al T}<r$, and so by 
Lemma \ref{EBL2} we have
$$
B_{g(t)}(x_0,R_0) \subset B_{g(s)}(x_0,R_0e^{\al (s-t)})
\subset B_{g(s)}(x_0,r).$$
Combining these two cases $s\in (0,t)$ and $s\in (t,T]$, we find that
$B_{g(t)}(x_0,R_0) \subset \Om_T$,
thus implying the first inclusions 
of \eqref{ball_inclusions}.

An immediate consequence of this inclusion is that 
$B_{g(t)}(x_0,\frac{R+R_0}{2}) \subset\subset \Interior(\Om_T)$, because
$\frac{R+R_0}{2}<R_0$.
Therefore the second inclusion of \eqref{ball_inclusions} will follow immediately if we can show that $B_{d_0}(x_0,R)\subset B_{g(t)}(x_0,\frac{R+R_0}{2})$ for some $t\in (0,T]$.
However, this will be true for sufficiently small $t>0$ by the uniform convergence of 
$d_{g(t)}$ to $d_0$ as $t\downto 0$.
\end{proof}

An optimised version of  estimate \eqref{lower_dist_bd_holder2} would tell us that we can take $\ga$ as close as we like to $1$
by taking $c_0$ sufficiently small.

\section{{Proof of the Local Existence Theorem \ref{main_exist_thm2}}}

We shall need the following generalisations and  consequences  of the Local Lemma, \cite[Lemma 2.1]{ST1}, the Double Bootstrap Lemma,   \cite[Lemma 9.1]{ST1} and Hochard's Lemma,  \cite[Lemma 6.2]{hochard}, in order to prove Theorem \ref{main_exist_thm2}.
In all of the three lemmata appearing below,
the Riemannian manifolds $(N,g)$ appearing are not necessarily complete. 
However, it still makes sense to define the injectivity radius at $p\in N$ as the supremum of the radii $r$ for which the exponential map at $p$ is well-defined on  the ball of radius $r$ in $T_pN$, and is a diffeomorphism from that ball to its image.

\begin{lemma}[{cf. \cite[Lemma 2.1]{ST1}}]
\label{modifiedlocal2}
Let $(N^3,g(t))_{t\in [0,T]}$ be a smooth Ricci flow such that for some fixed $x\in N$ we have $B_{g(t)}(x,1)\subset\subset N$ for all $t\in [0,T]$, and so that
\begin{compactenum}[(i)]
\item
$\VolB_{g(0)}(x,1) \geq v_0>0 $, and 
\item
$\Ric_{g(t)} \geq -1$ on $ B_{g(t)}(x,1) $ for all $t\in [0,T]$. 
\end{compactenum}
Then there exist $C_0 = C_0(v_0) \geq 1 $ and $\hat T = \hat T(v_0)>0$ such that 
$ |\Rm|_{g(t)}(x) \leq 
C_0/t
$,  and
$\inj_{g(t)}(x) \geq \sqrt{
t/C_0
}$ 
for all $0<t \leq \min(\hat T,T).$
\end{lemma}

\begin{proof}
Lemma 2.1 of \cite{ST1} tells us that
there exist $C_0 = C_0(v_0)\geq 1$, $\hat T = \hat T(v_0)>0$ and $\eta_0=\eta_0(v_0)>0$  such that 
$ |\Rm|_{g(t)} \leq \frac{C_0}{t}$ on $B_{g(t)}(x,1/2)$ 
and $\VolB_{g(t)}(x,1/2) \geq \eta_0$ for all $t \in (0,\min (\hat T,T)]$.
(The lemma there gave the volume of the unit ball, but we are assuming a Ricci lower bound throughout the unit ball, so Bishop-Gromov applies.)
The injectivity radius estimate of Cheeger-Gromov-Taylor \cite{CGT} and the Bishop-Gromov comparison principle then tell us (after scaling $g(t)$ to $\ti g = \frac{1}{t}g(t)$ for each $t$, and then scaling back), that there exists  $i_0 = i_0(\eta_0,C_0) = i_0(v_0)>0$ such that
$ \inj_{g(t)}(x) \geq i_0 \sqrt{t}$. 
By increasing $C_0$ if necessary, we may assume without loss of generality that $ i_0 \geq \frac{1}{\sqrt{C_0}}$. 
\end{proof}

\begin{lemma}[{cf. Double Bootstrap Lemma \cite[Lemma 9.1]{ST1}}]
\label{modifieddouble2}
Let $(N^3,g(t))_{t\in [0,T]}$ be a smooth Ricci flow, and $x\in N$, such that $ B_{g(0)}(x,2)$ is compactly contained in $N$ 
and so that throughout $B_{g(0)}(x,2) $ we have
\begin{compactenum}[(i)]
\item
$|\Rm|_{g(t)} \leq \frac{c_0}{t}$  for some $c_0\geq 1$ and all $t\in (0,T]$, and 
\item
$\Ric_{g(0)} \geq -\de_0$ for some $\de_0>0$.
\end{compactenum}
Then there exists $\hat S = \hat S(c_0,\de_0)>0$ such that 
$ \Ric_{g(t)}(x) \geq - 100 \de_0 c_0 $ 
for all $0\leq t \leq \min(\hat S,T).$
\end{lemma}

\begin{proof}
Using 
Lemma \ref{SBL}, we see that $B_{g(t)}(x,2-\beta \sqrt{c_0 t}) \subset B_{g(0)}(x,2) $ for some universal constant $\beta\geq 1$,
for all $t\in [0, T]$.
Choosing $\hat S = \frac{1}{\beta^2 c_0}$, we see 
that $ B_{g(t)}(x,1) \subset  
B_{g(0)}(x,2) $ for all $0\leq t \leq \min( \hat S,T)$.
The conclusions of the lemma now follow from direct application of 
\cite[Lemma 9.1]{ST1}, after decreasing $\hat S$ again if necessary.
\end{proof}

The proof of Theorem \ref{main_exist_thm2} will involve  
defining a local smooth solution $(B_{g_0}(x_0, r_1),g(t))$
  for $t \in [0,t_1]$ for some small (uncontrolled) $t_1>0$, and  $r_1 = s_0-1$, and then inductively defining smooth 
local extensions $(B_{g_0}(x_0, r_i) ,g(t))$
 for $t \in [0,t_i]$  where $t_i$ is a geometrically increasing sequence and 
$r_i$ is a decreasing sequence that nevertheless enjoys a good lower bound.
The general strategy of construction of a flow on dyadic time intervals was first used by
Hochard \cite{hochard}.
However, in our approach the solutions will each satisfy both $|\Rm|_{g(t)} \leq \frac{c_0}{t}$  and $\Ric_{g(t)} \geq -\al(v_0,\al_0) $
on $B_{g_0}(x_0, r_i)$ for all $t \in [0,t_i]$. 
 The main inductive step is achieved through the  Extension Lemma \ref{ExtensionLemma2} below, whose proof involves conformally modifying the metric, as in \cite[Lemma 6.2]{hochard}, and then using the two lemmata from above. 
We rewrite Hochard's Lemma here in a scaled form  for ease of application in the proofs that follow.

\begin{lemma}[{Variant of Hochard,  \cite[Lemma 6.2]{hochard}}]
\label{HochardsLemma2}
Let $(N^n,g)$ be a smooth (not necessarily complete) Riemannian manifold and let $U \subset N$ be an open set. 
Assume that for some $\rho\in (0, 1]$, we have
$\sup_{ U } |\Rm|_g  \leq \rho^{-2}$,
$B_{g}(x,\rho ) \subset \subset N$ and 
$\inj_{g}(x) \geq \rho$ for all $x \in U$. 
Then  there exist a constant $\ga=\ga(n)\geq 1$, 
an open set $\ti U \subset  U$ and 
a smooth metric $\ti g$ defined on $\ti U$ such that 
each connected component of
$(\ti U, \ti g)$ is a complete Riemmanian  manifold satisfying
\begin{compactenum}[\ (1)]
\item
$\sup_{\ti U} |\Rm|_{\ti g} \leq  \ga\rho^{-2}$
\item
$U_{\rho} \subset \ti U \subset  U$
\item
$\ti g= g  \ \mbox{\rm on} \  \ti U_{\rho} \supset U_{2\rho}$,
\end{compactenum} 
where  
$U_s = \{ x \in U \ | \ B_{g}(x,s) \subset \subset U \}$.
\end{lemma}

\bcmt{be aware that we allow sets to be empty sets above, if $U$ is very small}

\bcmt{Can't just restrict $\ti U$ to one connected component or 
the inclusion $U_{\rho} \subset \ti U$ will fail}

The strategy of Hochard to prove this lemma is to conformally blow up the metric in a neighbourhood of the boundary of $U$ so that it looks essentially hyperbolic. Once we have a complete metric locally, we can run the Ricci flow with standard existence theory.
The idea of cutting off a metric locally and replacing it with a complete hyperbolic metric in order to start the flow was introduced in \cite{JEMS} in a much simpler situation. A global conformal deformation of metric that is related to Hochard's construction was carried out in Section 8, in particular Theorem 8.4, of \cite{MilesCrelle3D}.


\begin{proof}
Scale the metric $g$ by $h= \rho^{-2} g$.
The new metric $h$ satisfies 
$ \sup_{U} |\Rm|_h \leq 1$, $B_h(x,1) \subset \subset N$ and
$\inj_{h}(x) \geq 1$  for all $ x \in U$.
Let $C(n)$ be the constant from Hochard's Lemma 6.2 in \cite{hochard}. Without loss of generality this  constant satisfies $C(n)>1$ since otherwise we can set $C(n)$  to be the maximum of the old $C(n)$ and $2$, and note that the conclusions of that lemma  will still be correct.
We use  Hochard's Lemma 6.2 \cite{hochard} applied to the Riemannian manifold $(N,h)$ and the set $U$ appearing in the statement of this lemma, with the choice of $k =C^2(n)$.  We conclude that 
there exists an open  set $\ti U \subset U$ and a metric $\ti h$ defined on $\ti U$ such that  each connected component of $(\ti U,\ti h)$ is smooth and complete  and satisfies
\begin{eqnarray}
&& (1) \ \ \sup_{\ti U} |\Rm|_{\ti h} \leq \ga \\
&& (2) \ \ U_1 \subset \ti U \subset  U \\
&& (3)\ \  \ti h = h  \ \mbox{on} \  \ti U_{1}, 
\end{eqnarray} where
$U_s = \{ x \in U \ | \ B_{h}(x,s) \subset \subset U \},$ and $\ga := C^2(n)$. 
Let $x$ be any point in $U_2$. Then $B_{h}(x,2) \subset \subset U$ 
by definition, and hence, by the triangle inequality, $B_{h}(x,1) \subset \subset U_1  \subset \ti U$, and thus $x \in \ti U_1$.
This shows that
$U_2 \subset \ti U_1,$  in view of the fact that $x \in U_2$ was arbitrary. Hence, the conclusion $(3)$ above  may be replaced by 
$ \ti h = h  \ \mbox{on} \  \ti U_{1} \supset U_2$.
Scaling back, that is, defining 
$\ti g =  \rho^2\ti h$, completes the proof.
\end{proof}

The Extension Lemma \ref{ExtensionLemma2} shows us how we can extend a smooth solution for a short, but well-defined time, still maintaining bounds like $|\Rm|_{g(t)} \leq \frac{c_0}{t}$, if the initial Ricci curvature and volume are bounded from below.  The cost is that we have to decrease the size of the region on which the smooth solution is defined. The extension will only be possible if the size of the radius $r_1$ of the ball (with respect to  $g(0)$) where the solution is defined is not too small, $r_1 \geq 2$ will suffice, and the time for which the solution is defined is not too large. 

In the following lemma, no metrics are assumed to be complete.

\newcommand{\hatTtau}{{\tau}}

\begin{lemma}[Extension Lemma]
\label{ExtensionLemma2}
For $v_0>0$ given, there exist  $c_0\geq 1$ and $\hatTtau>0$
such that the following is true.
Let $r_1 \geq 2$, and $(M,g_0)$ be a smooth three-dimensional Riemannian manifold such that  $B_{g_0}(x_0,r_1) \subset \subset M$, and
\begin{compactenum}[(i)]
\item
$\Ric_{g_0} \geq -\al_0$ for some $\al_0  \geq 1$ on $B_{g_0}(x_0,r_1)$, and  
\item
$\VolB_{g_0}(x,r) \geq v_0 r^3$ for all $r\leq 1$ and all $x \in B_{g_0}(x_0,r_1-r).$
\end{compactenum}
Assume further that we are given a smooth Ricci flow  $(B_{g_0}(x_0,r_1),g(t))$, $t \in [0,\ell_1]$, where 
$\ell_1 \leq \frac{\hatTtau}{200\al_0 c_0}$,
with $g(0)$ equal to the restriction of $g_0$, for which
\begin{compactenum}[(a)]
\item
$ |\Rm|_{g(t)} \leq \frac{c_0}{t}$  and
\item
$\Ric_{g(t)}\geq -\frac{\hatTtau}{\ell_1}$
\end{compactenum}
on $B_{g_0}(x_0,r_1)$  for all $t \in (0,\ell_1]$.
Then, setting $\ell_2  = \ell_1(1+\frac{1}{4c_0})$ and 
$r_{2} = r_{1} - 6\sqrt{\frac{\ell_{2}}{\hatTtau}}\geq 1$,
the Ricci flow $g(t)$ can be extended smoothly to a Ricci flow on
the smaller ball $B_{g_0}(x_0, r_{2})$, for the longer time interval $t\in [0,\ell_2]$, with 
\begin{compactenum}[(a$'$)]
\item
$ |\Rm|_{g(t)} \leq \frac{c_0}{t}$  and
\item
$\Ric_{g(t)}\geq -\frac{\hatTtau}{\ell_2}$
\end{compactenum}
throughout $B_{g_0}(x_0,r_2)$  for all $t \in (0,\ell_2]$.
\end{lemma}

A version of the Extension Lemma can be given without assumption $(b)$ 
since we can always obtain a lower Ricci bound by application of the Double Bootstrap Lemma \ref{modifieddouble2}. However, in practice we will always have this lower bound already, and its inclusion simplifies the proof and emphasises the natural symmetry between the hypothesis and conclusion.



\begin{proof}
For the given $v_0$, let  $C_0\geq 1$ and $\hat T>0$  be the constants given by  Lemma \ref{modifiedlocal2}.
With this choice of $C_0$  we choose $c_0 = 4\ga C_0>C_0$, where 
$\ga\geq 1$ 
is the constant coming from Hochard's Lemma \ref{HochardsLemma2} above,
and set $\de_0 = \frac{1}{ 100 c_0}$,
and we let $\hat S$ be the constant we obtain from 
Lemma \ref{modifieddouble2} for these choices.
We define $\hatTtau:=\min\{\hat T,\hat S\}$ so that 
we are free to apply 
Lemma \ref{modifiedlocal2} and 
Lemma \ref{modifieddouble2} with $\hat T$, respectively $\hat S$, replaced by $\hatTtau$.
We also reduce $\hatTtau$ if necessary so that 
\beq
\label{c0hatT_assumps2}
\be^2c_0\hatTtau\leq 1,
\qquad\hatTtau\leq 1, 
\qquad\text{ and }\hatTtau\leq C_0/4,
\eeq
where $\be\geq 1$ is the constant from Lemma \ref{SBL}.
The constants are now fixed.
Note that 
\beq
\label{ell_constraints2}
\ell_1\leq \ell_2\leq 2\ell_1\leq \hatTtau\leq 1.
\eeq

\bcmt{let's record where we need some of these constraints:
We use $\be^2c_0\hatTtau\leq 1$ and $\hatTtau\leq C_0$ in claim 1. 
The first of these we also need in claim 2.
We additionally need $\hatTtau\leq 1$ when doing Hochard so we can be sure that $\ell_1\leq 1$. We need $\hatTtau\leq C_0/4$ in claim 2.
We use $2\ell_1\geq \ell_2$ when getting the curvature estimates on the extension flow using the doubling time estimate, and also right at the end when getting the lower Ricci bound.}

{\bf Claim 1:}
For all $x\in U :=  B_{g_0}(x_0,r_1 - 2\sqrt{\frac{ \ell_1}{\hatTtau}})$,
we have 
$B_{g(t)}(x,\sqrt{
t/C_0
} )\subset \subset B_{g_0}(x_0,r_1)$,  
$\inj_{g(t)}(x) \geq \sqrt{
t/C_0
}$ 
and $|\Rm|_{g(t)}(x)\leq 
C_0/t
$,
for all $t\in (0,\ell_1]$.


Note that by assumption, we have $c_0/t$ curvature decay; the claim improves this to $C_0/t$ curvature decay, albeit on a smaller ball, as well as obtaining an injectivity radius bound. The original $c_0/t$ decay will nevertheless be required to control the nesting of balls.

\emph{Proof of Claim 1:}
For $x \in B_{g_0}(x_0,r_1- 2\sqrt{\frac{\ell_1}{\hatTtau}})$, the triangle inequality implies that  $B_{g_0}(x,2\sqrt{\frac{\ell_1}{\hatTtau}}) \subset \subset B_{g_0}(x_0,r_1)$ and hence by hypothesis, $|\Rm|_{g(t)} \leq \frac{c_0}{t}$ 
and $\Ric_{g(t)}\geq -\frac{\hatTtau}{\ell_1}$
on $B_{g_0}(x,2\sqrt{\frac{\ell_1}{\hatTtau}})$ for all $t \in (0,\ell_1]$.  Scaling the solution to $\hat  g(t) := \frac{\hatTtau}{ \ell_1 } g(t \frac{\ell_1}{ \hatTtau})$ we see that we have a solution $\hat  g(t)$
on $B_{g_0}(x_0,r_1) \supset\supset B_{\hat  g(0)}(x,2)$,
$t \in[0, \hatTtau]$ with $|\Rm|_{\hat  g(t)}\leq \frac{c_0}{t}$  
and $\Ric_{\hat g(t)}\geq -1$
on $B_{\hat  g(0)}(x,2)$ for all $t\in (0,\hatTtau]$.

Applying 
Lemma \ref{SBL}
to $\hat g(t)$, we find that
$B_{\hat g(t)}(x,2-\be\sqrt{c_0 t})\subset 
B_{\hat g(0)}(x,2)$ for all $t\in [0,\hatTtau]$,
and in particular, $B_{\hat g(t)}(x,1)\subset 
B_{\hat g(0)}(x,2)$ because of \eqref{c0hatT_assumps2}.

This puts us in a position to apply
Lemma \ref{modifiedlocal2} to $\hat g(t)$, 
giving that $\inj_{\hat  g(t)}(x) \geq \sqrt{
t/C_0
}$
and $|\Rm|_{\hat g(t)}(x)\leq C_0/t$
for all $0< t\leq \hatTtau$.
Scaling back, we see that $B_{g(t)}(x,\sqrt{
\ell_1/\hatTtau
} )\subset \subset B_{g_0}(x_0,r_1)$,  
$\inj_{g(t)}(x) \geq \sqrt{
t/C_0
}$ and 
and $|\Rm|_{g(t)}(x)\leq C_0/t$
for all $t\in (0,\ell_1]$,
which is a little stronger than
Claim 1 because $t\leq \ell_1$ and $C_0\geq \hatTtau$ by \eqref{c0hatT_assumps2}, so
$\frac{t}{C_0}\leq \frac{\ell_1}{\hatTtau}$.
\hfill$//$

\vskip 0.5cm

Claim 1, specialised to $t=\ell_1$, puts us in exactly the situation we require in order to apply Lemma \ref{HochardsLemma2} with
$U =  B_{g_0}(x_0,r_1 - 2\sqrt{\frac{ \ell_1}{\hatTtau}})$, $N =  B_{g_0}(x_0,r_1)$,
$g = g(\ell_1)$ and 
$\rho^2 := \frac{\ell_1}{C_0}\leq 1$ 
(recall \eqref{ell_constraints2}).
The output is a new, possibly disconnected, smooth manifold $ (\ti U, \ti g)$, each component of which is complete,
 such that 
 \begin{eqnarray}
  &&(1)\ |\Rm|_{\ti g} \leq \frac{\ga C_0}{\ell_1}=\frac{c_0}{4\ell_1} \mbox{ on } \ti U \cr
 &&(2)\ \ U_{\sqrt{\frac{\ell_1}{C_0}} } \subset \ti U \subset U \cr
 && (3) \ \ \ti g = g(\ell_1) \mbox{ on }  
 \ti U_{\sqrt{\frac{\ell_1} { C_0 } } }
 \supset U_{2\sqrt{\frac{\ell_1} { C_0 } } },
 \end{eqnarray}
where $U_r = \{ x \in U \ | B_{g(\ell_1)}(x,r) \subset\subset U \}$.

Before restarting the flow with $\ti g$ on one component of $\ti U$, we take a closer look at where $g(\ell_1)$ equals $\ti g$:


{\bf Claim 2:}
We have $B_{g_0}(x_0,r_1 - 4\sqrt{\frac{ \ell_1}{\hatTtau}})\subset U_{2\sqrt{\frac{\ell_1} { C_0 } } }$, where the metrics $g(\ell_1)$ and $\ti g$ agree.

\emph{Proof of Claim 2:}
By definition of $U$, for every $x\in B_{g_0}(x_0,r_1 - 4\sqrt{\frac{ \ell_1}{\hatTtau}})$, we have 
$B_{g_0}(x,2\sqrt{\frac{ \ell_1}{\hatTtau}})\subset\subset U$.
By assumption, we have $|\Rm|_{g(t)}\leq c_0/t$ on 
$B_{g_0}(x_0,r_1)$, and hence on $U$ and on 
$B_{g_0}(x,2\sqrt{\frac{ \ell_1}{\hatTtau}})$ for all $t\in (0,\ell_1]$, and so 
Lemma \ref{SBL}
tells us that
$B_{g_0}(x,2\sqrt{\frac{ \ell_1}{\hatTtau}})\supset
B_{g(t)}(x,2\sqrt{\frac{ \ell_1}{\hatTtau}}-\be\sqrt{c_0 t})$
for all $t\in [0,\ell_1]$.
Specialising to $t=\ell_1$, and recalling that 
$\be\sqrt{c_0 \ell_1}\leq \sqrt{\ell_1/\hatTtau}$ by \eqref{c0hatT_assumps2}, we see that
$B_{g(\ell_1)}(x,\sqrt{\frac{ \ell_1}{\hatTtau}})\subset\subset U$,
and by \eqref{c0hatT_assumps2} 
this gives
$B_{g(\ell_1)}(x,2\sqrt{\frac{ \ell_1}{C_0}})\subset\subset U$
as required to establish that $x\in U_{2\sqrt{\frac{\ell_1} { C_0 } } }$.
\hfill$//$

\bcmt{used $\hatTtau\leq C_0/4$ of \eqref{c0hatT_assumps2} in this last bit}

We now restart the flow at time $\ell_1$, using Shi's complete bounded-curvature Ricci flow, starting at the connected component of 
$(\ti U,\ti g)$ that contains $x_0$.
By Claim 2 (and the inequality $\ell_1<\ell_2$) the new Ricci flow will live on a superset of 
$B_{g_0}(x_0,r_1 - 4\sqrt{\frac{ \ell_1}{\hatTtau}})
\supset
B_{g_0}(x_0,r_1 - 4\sqrt{\frac{ \ell_2}{\hatTtau}})$, and we call it still $g(t)$ for $t$ beyond $\ell_1$.
By the standard \emph{doubling time estimate} \cite[Lemma 6.1]{CLN}, any smooth complete bounded-curvature Ricci flow $h(t)$ 
such that 
$|\Rm|_{h(0)} \leq K$ must satisfy $|\Rm|_{h(t)} \leq 2K$ for 
$t\leq \frac{1}{16K}$.
In our situation, where $K=\frac{c_0}{4\ell_1}$, this tells us that $g(t)$ will exist
for $t\in [\ell_1,\ell_1(1+\frac{1}{4c_0})]=[\ell_1,\ell_2]$ and satisfy
$|\Rm|_{g(t)}\leq \frac{c_0}{2\ell_1}\leq \frac{c_0}{\ell_2}
\leq\frac{c_0}{t}$ throughout its domain of definition by \eqref{ell_constraints2}. 
Thus, we have constructed a smooth extension to our Ricci flow on
$B_{g_0}(x_0,r_1 - 4\sqrt{\frac{ \ell_2}{\hatTtau}})$, 
now existing for $t\in [0,\ell_2]$,
and this satisfies
$|\Rm|_{g(t)}\leq\frac{c_0}{t}$.

It remains to show that by reducing the radius of our $g_0$ ball where the extension is defined to $r_2=r_{1} - 6\sqrt{\frac{\ell_{2}}{\hatTtau}}$, we can be sure of the lower Ricci bound as claimed in the lemma. Pick an arbitrary
$x\in B_{g_0}(x_0, r_{2})$
at which we would like to establish the lower bound
$\Ric_{g(t)}\geq -\frac{\hatTtau}{\ell_2}$.
Observe that the extended Ricci flow is defined throughout
$B_{g_0}(x,2\sqrt{\frac{\ell_2}{\hatTtau}})
\subset\subset
B_{g_0}(x_0,r_1 - 4\sqrt{\frac{ \ell_2}{\hatTtau}})$ for all $t\in [0,\ell_2]$.

We scale up so that $\ell_2$ goes to $\hatTtau$. That is we define $\hat g(s) := \frac{\hatTtau}{\ell_2} g( \frac{s \cdot  \ell_2}{\hatTtau})$.
Then we have, for this scaled solution, 
$B_{\hat g(0)}(x,2)$ compactly contained within the domain of definition of the flow, for all $t\in [0,\hatTtau]$, and  $|\Rm|_{g(t)} \leq \frac {c_0}{t}$ on  $B_{\hat g(0)}(x,2)$  for all $t\in (0,\hatTtau]$.

Moreover, we have $\Ric_{\hat  g(0)} \geq -\de_0$ on 
$B_{\hat  g(0)}(x,2)$ 
for the choice of $\de_0$ made above, because keeping in mind
that $\ell_1 \leq \frac{\hatTtau}{200\al_0 c_0}$, we have
$$\Ric_{\hat  g(0)}
\geq -\al_0 \frac{\ell_2}{ \hatTtau}
\geq -2\al_0 \frac{\ell_1}{ \hatTtau}
\geq -\frac{1}{100c_0}
= -\de_0.$$
Using Lemma \ref{modifieddouble2}  with the $c_0$ and $\de_0$ 
we have chosen, 
we see that 
$\Ric_{\hat  g(t)}(x) \geq -1$ for all $t\in [0,\hatTtau]$.
Rescaling back, we see that 
$\Ric_{g(t)}(x) \geq -\frac{\hatTtau}{\ell_2}$ for all $t\in [0,\ell_2]$
as required to complete the proof.
\end{proof}

\brmk
\label{end_time_controlled_rmk}
With hindsight, we see that $(B_{g_0}(x_0,r_2),g(\ell_2))$ is the restriction to a local region 
of the final time of a complete Ricci flow $g(t)$ for $t\in [\ell_1,\ell_2]$ with the property that
$|\Rm|_{g(t)}\leq c_0/\ell_1$. Global derivative bounds for such flows, applied over the time interval $[\ell_1,\ell_2]$, give us a bound $|\grad^k\Rm|_{g(\ell_2)}\leq C/\ell_2^{1+k/2}$
at the end time.
Moreover, this flow $g(t)$ has initially, at $t=\ell_1$, a lower injectivity radius bound, and such a bound will remain at time $\ell_2$ because of the curvature bound for the flow, giving 
$\inj_{g(\ell_2)}\geq \sqrt{\ell_2/C}$.
Thus at the \emph{end} time of the extension, the extended metric is isometrically embedded within a complete Riemannian manifold with good bounds on its curvature and all its derivatives, as well as its injectivity radius.
\ermk

\begin{proof}[{Proof of Theorem \ref{main_exist_thm2}}]
We may assume, without loss of generality, that $\al_0 \geq 1$: if $\al_0 <1$ 
then
$\Ric_{g_0} \geq -\al_0 $ implies $\Ric_{g_0} \geq -1$ and so we replace $\al_0$ in this case by $1$.
From the Bishop-Gromov comparison principle, by reducing $v_0$ to a smaller positive number depending on $\al_0$ and the original $v_0$, we may assume without loss of generality that
$\VolB_{g_0}(x,r) \geq v_0 r^3$ for all $x \in  B_{g_0}(x_0,s_0-1)$ and for all $r\in (0, 1]$. 
Let $c_0$ and $\hatTtau$ be the constants given by
Lemma \ref{ExtensionLemma2} for this new $v_0$.

Since $B_{g_0}(x_0,s_0)$ is  compactly contained in  $M$, we can be sure that
$\sup_{B_{g_0}(x_0,s_0)} |\Rm|_{g_0} \leq \rho^{-2}< \infty$,
and $B_{g_0}(x,\rho)\subset\subset M$ 
and $\inj_{g_0}(x)\geq \rho$ for all 
$x\in B_{g_0}(x_0,s_0)$,
for some $\rho\in (0,\half]$ depending on $(M,g_0)$, $x_0$ and $s_0$.
By using Hochard's Lemma, Lemma \ref{HochardsLemma2} of this paper, 
with $U:=B_{g_0}(x_0,s_0)$,
we can find a connected subset $\ti M\subset U \subset M$ containing 
$B_{g_0}(x_0,s_0-\half)$,  
and a 
smooth, complete metric $\ti g_0$ on $\ti M$
with $\sup_{\ti M} |\Rm|_{\ti g_0} < \infty$  such that on 
$B_{g_0}(x_0,s_1)$, where $s_1 := s_0-1\geq 3$, the metric remains unchanged, i.e. $\ti g_0 = g_0$ there.
By renaming $(\ti M,\ti g_0)$ as $(M,g_0)$ we have reduced the theorem to the following:

{\bf Claim:} Suppose that
$(M^3,g_0)$ is a \emph{complete} Riemannian manifold with 
\emph{bounded curvature}, $x_0\in M$,
$s_1\geq 3$, and 
\begin{equation}
\left\{
\begin{aligned}
& \Ric_{g_0}\geq -\al_0\leq -1 \qquad & &\text{on } B_{g_0}(x_0,s_1)\\
& \VolB_{g_0}(x,r)\geq v_0 r^3>0 \qquad & &\text{for all }r\in (0,1] \text{ and }x\in B_{g_0}(x_0,s_1-r).
\end{aligned}
\right.
\end{equation}
Then there exist $T,\al,c_0>0$ depending only on $\al_0$ and $v_0$,
and a Ricci flow $g(t)$ defined for $t\in [0,T]$ on $B_{g_0}(x_0,s_1-1)$, with
$g(0)=g_0$ where defined, such that
\begin{equation}
\left\{
\begin{aligned}
& \Ric_{g(t)}\geq -\al \qquad & &\text{on } B_{g_0}(x_0,s_1-1)\\
& |\Rm|_{g(t)}\leq c_0/t \qquad & &\text{on } B_{g_0}(x_0,s_1-1)
\end{aligned}
\right.
\end{equation}
for all $t\in (0,T]$.

Now that we have reduced to the case that 
$(M,g_0)$ is complete, with bounded curvature, we can take Shi's Ricci flow: There exists a smooth, complete, bounded-curvature Ricci flow $g(t)$ on $M$ for some nontrivial time interval $[0,t_1]$ with $g(0)=g_0$.
In view of the boundedness of the curvature, after possibly reducing $t_1$ to a smaller positive value, we may trivially assume that
$|\Rm|_{g(t)} \leq  \frac{c_0}{t}$ for all $t\in (0,t_1]$
and $\Ric_{g(t)}\geq -\frac{\hatTtau}{t_1}$ for all $t\in [0,t_1]$.

\bcmt{In Remark \ref{tau_j_rmk}, we argue that we can get more information if here we reduce $t_1$ if necessary so that $t_1=\nu^{2j}$ for some $j\in \N$, where $\nu$ is as below}

Of course, what is lacking from our flow is uniform control on its existence time. 
If $ t_1 \geq \frac{\hatTtau}{200\al_0 c_0}$, then we \emph{do} have such control, but otherwise we will be able to iteratively apply
the Extension Lemma \ref{ExtensionLemma2} in a manner analogous to that employed by Hochard \cite{hochard}, to get successive extensions 
until we have a flow defined for a uniform time.

Having found $t_1$ and $s_1$, for $i\in \N$ we define $t_i$ and $s_i$ iteratively by setting 
$t_{i+1}=\nu^{-2} t_i$, where $\nu:=(1+\frac{1}{4c_0})^{-1/2}\in (0,1)$, and 
$s_{i+1}=s_i-\mu \sqrt{t_{i+1}}$, where $\mu:=6\hatTtau^{-1/2}$.
Thus 
\beqa
s_i &= s_{i-1}-\mu\sqrt{t_i}\\
&= s_{i-2}-\mu\sqrt{t_i}-\mu\sqrt{t_{i-1}}\\
&= s_{i-2}-\mu\sqrt{t_i}[1+\nu]\\
&= s_{1}-\mu\sqrt{t_i}[1+\nu+\cdots +\nu^{i-2}]\\
&> s_1-\frac{\mu}{1-\nu}\sqrt{t_i}.
\eeqa
Our iterative assertion, that we have established above for $i=1$ is:
$$I(i)\left\{\begin{tabular}{l}
We have constructed a smooth Ricci flow $g(t)$ on 
$B_{g_0}(x_0,s_i)$ for $t \in [0,t_i]$, \\ 
with $g(0)=g_0$ on this ball, such that
on $B_{g_0}(x_0,s_i)$, and for $t \in (0,t_i]$, we have \\
$ |\Rm|_{g(t)} \leq \frac{c_0}{t} $ 
and
$\Ric_{g(t)}\geq -\frac{\hatTtau}{t_i}.$
\end{tabular}\right.$$

The Extension Lemma \ref{ExtensionLemma2}, applied with $r_1=s_i$ and $\ell_1=t_i$, tells us that $I(i)$ implies $I(i+1)$, provided that $t_i$ remains below 
$\frac{\hatTtau}{200\al_0 c_0}$, and $s_i$ does not get too small.
More precisely we iteratively apply that lemma until either
$t_i>\frac{\hatTtau}{200\al_0 c_0}$ or $t_{i+1}> \frac{(1-\nu)^2}{\mu^2}$. 
The latter requirement ensures that 
$\frac{\mu}{1-\nu}\sqrt{t_i}\leq 1$ for $i\geq 2$ up to when we stop iterating, which in turn ensures that $s_i\geq s_1-1\geq 2$ as required for the radius $r_1$ in the extension lemma.
Either way, the final assertion $I(i)$ tells us that we have
constructed the desired Ricci flow $g(t)$ for the time interval 
$[0,t_i]$, which has positive length bounded below depending only on
$\al_0$ and $v_0$, defined on the domain
$B_{g_0}(x_0,s_i)$ where 
$s_i > s_1-\frac{\mu}{1-\nu}\sqrt{t_i}\geq s_1-1$.
This completes the proof of the claim, and hence of Theorem \ref{main_exist_thm2}.
\end{proof}

\section{{Proof of the Mollification theorem \ref{mollifier_thm}}}
\label{moll_pf_sect}

The existence assertion of Theorem \ref{mollifier_thm} will follow rapidly from the Local Existence Theorem \ref{main_exist_thm2}. 
Before we can apply that theorem, we must observe that a standard volume comparison argument tells us that there exists some smaller $v_0>0$ depending only on the original $v_0$, $\ep$ and $\al_0$, such that for all $x\in B_{g_0}(x_0,1-\frac{\ep}{4})$, we have $\VolB_{g_0}(x,\frac{\ep}{4})\geq v_0$.

We can then rescale the initial metric $g_0$ to $\ti g_0:=\frac{16}{\ep^2}g_0$,
i.e. expand distances by a factor of $4/\ep$, to put ourselves in exactly the situation of Theorem \ref{main_exist_thm2}, with $s_0=4/\ep\geq 40$, for some new $\al_0>0$ depending only on the old $\al_0$ and $\ep$.
The output of Theorem \ref{main_exist_thm2} is a Ricci flow on 
$B_{\ti g_0}(x_0,s_0-2)$ with estimates, and after returning to the original scaling, we have
a Ricci flow $g(t)$ on $B_{g_0}(x_0,1-\ep/2)$, for $t\in [0,T]$, where $T>0$ depends only on $\al_0$, $v_0$ and $\ep$, with $g(0)=g_0$ where defined, and so that 
\begin{equation}
\label{littly}
\left\{
\begin{aligned}
& \Ric_{g(t)}\geq -\al \\
& |\Rm|_{g(t)}\leq c_0/t 
\end{aligned}
\right.
\end{equation}
on $B_{g_0}(x_0,1-\ep/2)$,
for all $t\in (0,T]$, for some $\al,c_0>0$ depending only on $\al_0$, $v_0$ and $\ep$.


It remains to establish the desired properties of $g(t)$, and in order to do this we may have to reduce $T$ to a smaller positive number, with the same dependencies. Looking first at the desired curvature control of \eqref{curv_control_moll}, 
we observe that the first assertion was already obtained in
the second inequality of \eqref{littly} above on an even larger domain.
This larger domain is useful, however, since it allows us to 
invoke Shi's local derivative estimates to give the 
bounds for the higher derivatives claimed in the second inequality 
of \eqref{curv_control_moll}, as we now explain.
First we apply the Shrinking Balls Lemma \ref{SBL} to deduce that after possibly reducing $T>0$ depending on $c_0$ and $\ep$, we have $B_{g(t/2)}(x_0,1-\frac{2}{3}\ep)\subset B_{g_0}(x_0,1-\frac{\ep}{2})$, for all $t\in [0,T]$. Then, we reduce $T>0$ further if necessary, 
depending on $\al$ and $\ep$, to be sure that 
$B_{g(t/2)}(x_0,1-\frac{5}{6}\ep)\supset B_{g_0}(x_0,1-\ep)$ for all $t\in [0,T]$,
this time by Lemma \ref{EBL2}.
To obtain the required bounds for the higher derivatives at time $t$, we can then apply Shi's estimates over the time interval $[t/2,t]$ on
the ball $B_{g(t/2)}(x_0,1-\frac{2}{3}\ep)$.
A reference for Shi's estimates in almost the required form is \cite[Theorem 14.14]{chowRFTA_V2P2}, although one needs to observe that the constant $C(\al,K,r,m,n)$ in that theorem can be given as 
$C(\al,r,m,n)K$ by a simple rescaling argument, and indeed to rescale in order to obtain the estimates on the whole of $B_{g(t/2)}(x_0,1-\frac{5}{6}\ep)\supset B_{g_0}(x_0,1-\ep)$ at time $t$.

At this point we can restrict the Ricci flow to
$B_{g_0}(x_0,1-\ep)$.

Next, we apply the Shrinking Balls Lemma \ref{SBL} again to deduce that after possibly reducing $T>0$ depending on $c_0$ and $\ep$, we have $B_{g(t)}(x_0,1-2\ep)\subset B_{g_0}(x_0,1-\ep)$, for all $t\in [0,T]$.
In turn, this allows us to apply our  Lower Volume Control Lemma, \cite[Lemma 2.3]{ST1} with $\ga=1-2\ep$.
In order to do so, we first observe that by volume comparison, there is a positive lower bound for $\VolB_{g_0}(x_0,1-2\ep)$ depending only on $v_0$ and $\al_0$. (There is not even a dependence on $\ep$ because we are assuming $\ep\leq 1/10$.)
The output of that lemma is that after possibly reducing $T>0$ a little further, without adding any dependencies, there exists $v>0$ as claimed so that
$\VolB_{g(t)}(x_0,1-2\ep)\geq v$ for all $t\in [0,T]$, which is the remaining part of 
\eqref{moll_thm_conc1}.

Prior to addressing the claims on the distance function, we must verify that for $s,t\in [0,T]$, we have 
$B_{g(s)}(x_0,\half-2\ep)\subset B_{g(t)}(x_0,\half-\ep)$, provided that we restricted $T>0$ sufficiently, depending only on $\al_0$, $v_0$ and $\ep$.
To see this, observe that either we have
$s\in [t,T]$, in which case it follows from the Shrinking Balls Lemma \ref{SBL} (provided we restrict $T>0$ depending only on 
$c_0$ and $\ep$)
or we have $s\in (0,t)$, in which case 
it follows from the Expanding Balls Lemma \ref{EBL2} (and $T$ must be reduced also depending on $\al$ and $\ep$).

The final parts of the theorem concerning the distance function then follow immediately from
Lemma \ref{biholder_dist2} by setting $r=\half-\ep$, cf. Remark \ref{got0}, completing the proof of the theorem. 
$\hfill\Box$

\bcmt{In the remark below, note we're not claiming that we can get uniform bounds for every $t\in (0,T]$ on all derivatives. The problem is that as we restart in the extension lemma, the derivatives of the curvature are initially uncontrolled. Instead, to get this improved control, we need to make sure that $\tau_i$ is at the end of one of the dyadic time intervals, and Shi has had enough time to work. 
}

\brmk
\label{tau_j_rmk}
We are not claiming that the Ricci flow $(\wasB,g(t))$ we construct in this theorem is (isometric to) a restriction to a local region of a complete Ricci flow. However, by developing a little the statement and proof of Theorem \ref{main_exist_thm2}, one could add the conclusion 
in Theorem \ref{mollifier_thm} that
for $\nu:=(1+\frac{1}{4c_0})^{-1/2}\in (0,1)$ as used in the proof of Theorem \ref{main_exist_thm2}, if we define 
$\tau_j:=\nu^{2j}\downto 0$ then for sufficiently large $j$,
$(\wasB,g(\tau_j))$ can be isometrically embedded within a complete Riemannian manifold
$(\m_j^3,g_j)$ such that
\begin{equation}
\left\{
\begin{aligned}
& |\Rm|_{g_j}\leq C(\al_0,v_0,\ep)/{\tau_j} \\
& \left|\grad^k \Rm\right|_{g_j}\leq 
C(k,\al_0,v_0,\ep)/\tau_j^{1+\frac{k}{2}}
\end{aligned}
\right.
\end{equation}
\emph{globally} throughout $\m_j$, and so that
$$\inj_{g_j}(\m_j)\geq \eta\sqrt{\tau_j}$$
for some $\eta=\eta(\al_0,v_0,\ep)>0$.
See Remark \ref{end_time_controlled_rmk}, and note that in the proof of 
Theorem \ref{main_exist_thm2}, we may as well assume that $t_1=\nu^{2m}$
for some $m\in \N$, by reducing $t_1$ a little if necessary, and so
$t_i=\nu^{2(m+1-i)}$.
\ermk

\section{{Ricci limit spaces in 3D are bi-H\"older to smooth manifolds}}
\label{ricci_lims_sect}

In this section we prove Theorem \ref{preACCT}. The result considers a sequence of coarsely controlled manifolds and obtains compactness; a subsequence converges to an optimally-regular limit. In particular, the limit is much more regular than we learn from Gromov compactness.
The strategy is to regularise the coarsely controlled manifolds using Ricci flow, and in particular using the Mollification Theorem \ref{mollifier_thm}. Once we have regularity, then we have compactness, and a subsequence of the Ricci flows will converge essentially to a Ricci flow whose initial data represents the desired optimally-regular limit of the original sequence of manifolds.

A number of subtleties arise in the course of the proof, even once Theorem \ref{mollifier_thm} has been proved. Because we must work locally, and we only have uniform regularity for positive times, we have to take care over which region we try to extract a limit. We have no real choice other than to work in a time $t>0$ ball, but then we have to ensure that we end up with a limit that is defined on a time $0$ ball of positive radius. For this, we need the Ricci lower bound estimates of \cite{ST1}, as well as the $c_0/t$ decay of the full curvature tensor.

We also have to ensure that the limit Ricci flow has initial data (a metric space) that agrees with the limit of the original initial metrics. For this, we need strong uniform control on the evolution of the Riemannian distance, and again this requires our lower Ricci bounds and upper $c_0/t$ curvature decay.

It may be helpful to record a result ensuring that appropriate smooth \emph{local} convergence of Riemannian manifolds will imply convergence of the distance functions on a sufficiently smaller region.


\begin{lemma}[Convergence of distance functions under local convergence]
\label{dist_conv_lem}
Suppose $(M_i,g_i)$ is a sequence of smooth $n$-dimensional Riemannian manifolds, not necessarily complete,
and that $x_i\in M_i$ for each $i$. 
Suppose that there exist a smooth, possibly incomplete $n$-dimensional Riemannian manifold 
$(\n,\hat g)$ and a point $x_0\in \n$ with $B_{\hat g}(x_0,2r)\subset\subset\n$ for
some $r>0$, 
and a sequence of smooth maps $\vph_i:\n\to M_i$, diffeomorphic onto their images and mapping $x_0$ to $x_i$, such that 
$\vph_i^*g_i \to \hat g$ smoothly on 
$\overline{B_{\hat g}(x_0,2r)}$.
Then 
\begin{compactenum}
\item
If $0<a\leq 2r$, and $a<b$, then $\vph_i(B_{\hat g}(x_0,a))\subset
B_{g_i}(x_i,b)$ for sufficiently large $i$.
\item
If $0<a<b\leq 2r$, then $B_{g_i}(x_i,a)\subset 
\subset \vph_i(B_{\hat g}(x_0,b))$  for sufficiently large $i$.
\item
For every $s\in (0,r)$, we have convergence of the distance functions
$$d_{g_i}(\vph_i(x),\vph_i(y))\to d_{\hat g}(x,y)$$
as $i\to\infty$, uniformly as $x$ and $y$ vary within $B_{\hat g}(x_0,s)$.
\end{compactenum}
\end{lemma}

\bcmt{we've left the `Remark \ref{dist_rmk}' issues implicit in Part 3, which I think is OK now we have parts 1 and 2.}

\begin{proof}[{Proof of Lemma \ref{dist_conv_lem}}]
For Part 1, if $x\in B_{\hat g}(x_0,a)$, then we can take a minimising $\hat g$-geodesic $\si$ from $x_0$ to $x$, so that $\vph_i\circ\si$ is a path of length less than $b$ for sufficiently large $i$, independent of $x$.

For Part 2, first note that since we are free to adjust $a$ and $b$ a little, it suffices to prove an inclusion rather than a compact inclusion.
If the inclusion failed for every $i$ after taking a subsequence, then we could take a sequence of points 
$z_i\in B_{g_i}(x_i,a)$ not in
$\vph_i(B_{\hat g}(x_0,b))$. We could then take a smooth path 
$\si_i:[0,1]\to M_i$ connecting $x_i$ to $z_i$ with $g_i$-length no more than $a$, and thus lying within $B_{g_i}(x_i,a)$.
(We do not know at this stage whether $\si_i$ can be taken to be a minimising geodesic.)
By truncating $\si_i$, 
and modifying $z_i$ accordingly,
we may assume that $\si_i([0,1))\subset \vph_i(B_{\hat g}(x_0,b))$. The path $\vph_i^{-1}\circ\si_i$ must then have $\hat g$-length only a little more than the 
$g_i$-length of $\si_i$, and in particular
less than $b$, for sufficiently large $i$, 
which is a contradiction.

For Part 3, 
let $\de>0$ be arbitrary; we would like to show that for sufficiently large $i$,  we have
$$|d_{g_i}(\vph_i(x),\vph_i(y))- d_{\hat g}(x,y)|\leq \de,$$
for every $x,y\in B_{\hat g}(x_0,s)$.
For any such $x,y$, let $\si$ be a minimising geodesic with respect to $\hat g$ that connects these points, and 
note that $\si$ must remain within $B_{\hat g}(x_0,2s)\subset\subset
B_{\hat g}(x_0,2r)$, cf. Remark \ref{dist_rmk}. 
For each $i$ we map $\si$ forwards to $\vph_i\circ\si$. By the convergence $\vph_i^*g_i \to \hat g$, for sufficiently large $i$ (independent of the particular $x$ and $y$ we chose but depending on $r$) the length cannot increase by more than $\de$, i.e.
$$d_{g_i}(\vph_i(x),\vph_i(y))- d_{\hat g}(x,y)\leq \de.$$
To prove the reverse direction, 
set $s_1=(s+r)/2$, so $s<s_1<r$, 
and throw away finitely many terms in $i$ so that the image under $\vph_i$ of $B_{\hat g}(x_0,s)$ is contained in $B_{g_i}(x_i,s_1)$ (by Part 1).
Therefore, for any two points $x,y\in B_{\hat g}(x_0,s)$, any minimising geodesic
(with respect to $g_i$) connecting $\vph_i(x)$ and $\vph_i(y)$ 
must lie within the ball $B_{g_i}(x_i,2s_1)$,
and at least one such geodesic must exist because
$B_{g_i}(x_i,2s_1)\subset\subset  \vph_i(B_{\hat g}(x_0,2r))$
by Part 2.
%
%
%
Thus by considering the length of preimages under $\vph_i$ of minimising geodesics between $\vph_i(x)$ and $\vph_i(y)$, we see that
$$d_{\hat g}(x,y)-d_{g_i}(\vph_i(x),\vph_i(y))\leq \de,$$
for sufficiently large $i$, independent of $x$ and $y$ (but depending on $r$).
\end{proof}



\begin{proof}[{Proof of Theorem \ref{preACCT}}]
We begin by applying the Mollification Theorem \ref{mollifier_thm} for each $i$, with $\ep=1/100$ fixed. The result is a collection of positive constants $T,v,\al,c_0$ as in that theorem, and a sequence of Ricci flows $g_i(t)$, $t\in [0,T]$, defined on the balls $B_{g_i}(x_i,1-\ep)$, with $g_i(0)=g_i$ where defined, such that for all $t\in [0,T]$ we have
$B_{g_i(t)}(x_i,1-2\ep)\subset\subset B_{g_i}(x_i,1-\ep)$ where the Ricci flows are defined,
and  
\begin{equation}
\label{moll_thm_conc_app}
\left\{
\begin{aligned}
& \Ric_{g_i(t)}\geq -\al \qquad & &\text{on } 
B_{g_i}(x_i,1-\ep)\\
& \VolB_{g_i(t)}(x_i,1-2\ep)\geq v>0, & &\\
\end{aligned}
\right.
\end{equation}
while for all $t\in (0,T]$ we have
\begin{equation}
\left\{
\begin{aligned}
\label{curv_control_moll_app}
& |\Rm|_{g_i(t)}\leq c_0/t \\
& \left|\grad^k \Rm\right|_{g_i(t)}\leq 
C/t^{1+\frac{k}{2}}
\end{aligned}
\right.
\end{equation}
on $B_{g_i}(x_i,1-\ep)$ for any $k\in\N$, where
$C$ depends on $k$, $\al_0$ and $v_0$ (since $\ep$ has been fixed).
Moreover, Theorem \ref{mollifier_thm} also tells us that
for each $i\in \N$, $s\in [0,T]$ and 
$X,Y\in B_{g_i(s)}(x_i,\half-2\ep)$, we have
$X,Y\in B_{g_i(t)}(x_i,\half-\ep)$ for all $t\in [0,T]$ so 
the infimum length of curves  within 
$(B_{g_i}(x_i,1-\ep),g_i(t))$ 
connecting $X$ and $Y$
is realised by a geodesic within $B_{g_i(t)}(x_i,1-2\ep)\subset B_{g_i}(x_i,1-\ep)$ 
where $g_i(t)$ is defined.
Moreover, for any $t\in [0,T]$
we have
\beq
\label{main_dist_est_for_g_i}
d_{g_i}(X,Y)-\be\sqrt{c_0 t}
\leq d_{g_i(t)}(X,Y)
\leq e^{\al t}d_{g_i}(X,Y).
\eeq

During the proof it will be necessary to take these flows, and other flows with the same estimates, and argue that balls of a given radius $r\in (0,\half)$ at a given time lie within balls of slightly larger radius $r+\ep$ at any different time. For this to work, we will apply Lemmata \ref{EBL2} and \ref{SBL}, which will do the job provided that 
$T$ is small compared with $\ep$ (and $\ep^2$) depending also on $\al$ and $c_0$
(we have $n=3$). We will also need to be able to bound $R_0$ in an application of 
Lemma \ref{biholder_dist2}. 
With hindsight, it will be enough to reduce $T>0$ if necessary so that
\beq
\label{Talep}
T\leq \frac{\ep^2}{100c_0\be^2}\qquad\text{and}\qquad
T\leq\frac{\ep}{100\al},
\eeq
where $\be$ is from Lemma \ref{SBL}.

The curvature bounds coming from the Mollification Theorem \ref{mollifier_thm}  allow us to obtain compactness using the argument of Hamilton-Cheeger-Gromov. 
Our situation is a bit easier than theirs in that we have already argued that all the derivatives of the curvature are bounded, and as is standard, by differentiating the Ricci flow equation we also have uniform control on all space-\emph{time} derivatives of the curvature.
However, there is an extra subtlety arising from working only locally in that we have to carefully choose the region on which we work in order for the standard argument to go through verbatim. In particular, we do not work on a time $t=0$ ball since its geometry with respect to a time $t>0$ metric is uncontrolled, and the standard covering arguments used in the compactness theory would fail.

Instead we work on a time $t>0$ ball, and we choose to work at the final time $t=T$.
We have established above that the Ricci flows $g_i(t)$ are each defined on the ball
$B_{g_i(T)}(x_i,1-2\ep)$, for $t\in [0,T]$.
By Remark \ref{dist_rmk}, if we halve the radius, and consider the ball
$B_{g_i(T)}(x_i,\half-\ep)$, then the distance function is unambiguous and is realised by a minimising geodesic lying within the region $B_{g_i(T)}(x_i,1-2\ep)$ where the metric $g_i(T)$ is defined.
By our volume and curvature bounds, we can obtain compactness on, say, $\overline{B_{g_i(T)}(x_i,\half-2\ep)}$.
More precisely, 
after passing to a subsequence
there exist a (typically incomplete) manifold $(\n,g_\infty)$ and a point $x_0\in\n$ such that $B_{g_\infty}(x_0,\half-2\ep)\subset\subset \n$,
and a sequence of smooth maps
$\vph_i:\n\to 
B_{g_i(T)}(x_i,1-2\ep)\subset M_i$,
diffeomorphic onto their image, and mapping $x_0$ to $x_i$,
such that $\vph_i^* g_i(T)\to g_\infty$ smoothly on
$\overline{B_{g_\infty}(x_0,\half-2\ep)}$.

We observe that with our set-up, this compactness assertion follows from the usual proofs in the smooth case. As an alternative to requiring detailed knowledge of the proofs, we 
note that by Remark \ref{tau_j_rmk}, after reducing $T$ to make it an integral power of $\nu^2=(1+\frac{1}{4c_0})^{-1}$, we could extend Theorem \ref{mollifier_thm} and obtain that
in fact $(B_{g_i}(x_i,1-\ep),g_i(T))$ can be isometrically embedded in a complete Riemannian manifold with uniform ($i$-independent) bounds on all derivatives of the curvature and a positive uniform lower bound on the injectivity radius. Thus compactness can be obtained from the standard theorems, after which we can restrict to $\overline{B_{g_\infty}(x_0,\half-2\ep)}$ in the limit.

Once we have compactness at time $T$, Hamilton's original argument, using the uniform curvature bounds for positive time, allows us to pass to a further sequence in $i$ to obtain a smooth  Ricci flow $g(t)$ living on $\overline{B_{g_\infty}(x_0,\half-2\ep)}$ for $t\in (0,T]$ such that $\vph_i^*g_i(t)\to g(t)$ smoothly locally on 
$\overline{B_{g_\infty}(x_0,\half-2\ep)}\times  (0,T]$. (Here $g(T)=g_\infty$ where $g(T)$ is defined.)

\bcmt{we only construct the Ricci flow on this closed ball, not on $\n$. We are only claiming convergence of the metrics at time $T$ on this closed ball, so currently have no choice}

Because of the smooth convergence, the curvature estimates in \eqref{moll_thm_conc_app} and 
\eqref{curv_control_moll_app} pass to the limit, and we have
\begin{equation}
\label{curv_bds_in_limit}
\left\{
\begin{aligned}
& \Ric_{g(t)}\geq -\al \\
& |\Rm|_{g(t)}\leq c_0/t 
\end{aligned}
\right.
\end{equation}
on the whole domain $\overline{B_{g_\infty}(x_0,\half-2\ep)}$ for all $t\in (0,T]$.

\bcmt{Currently not passing volume bounds to the limit.}

\bcmt{$g(T)$ is easier to picture than $g_\infty$, but the latter is defined on the whole of $\n$}

We have constructed a Ricci flow on a reasonably-sized ball with respect to $g(T)$, but 
we have to be concerned that this region might be much smaller with respect to earlier metrics, and might even be contained within a $g(t)$ ball of radius $r(t)$ that converges to zero as $t\downto 0$.
We will be rescued from this possibility by the refinement of the Expanding Balls Lemma, given in Lemma \ref{EBL2};
because we reduced $T$ in \eqref{Talep}, that lemma turns our control on the Ricci tensor 
in \eqref{curv_bds_in_limit} into the inclusion
\beq
\label{initial_inclusion}
B_{g(t)}(x_0,\half-3\ep)\subset\subset B_{g(T)}(x_0,\half-2\ep),
\eeq
for any $t\in (0,T]$.
This tells us, via Remark \ref{dist_rmk}, that if 
$x,y\in B_{g(t)}(x_0,\frac14-\frac32\ep)$, the corresponding ball of half the radius, then the $g(t)$-distance  between $x$ and $y$ 
is realised by a minimising geodesic lying within the domain of definition of the Ricci flow $g(t)$ (see Remark \ref{dist_rmk}).
Also, we can apply Lemma \ref{biholder_dist2}
with $r=\frac14-\frac32\ep$ to give us an extension of $d_{g(t)}$ to a metric $d_0$
on $\Om_T$, as defined in that lemma, together with the distance estimates
\beq
\label{specialised_main_dist_est2_repeat}
d_0(x,y)-\be\sqrt{c_0t}
\leq d_{g(t)}(x,y)
\leq e^{\al t}d_0(x,y),
\eeq
and
\beq
\label{holder_est_repeat}
d_{g(t)}(x,y)\geq \ga \left[d_0(x,y)\right]^{1+4c_0},
\eeq
for all $t\in (0,T]$, $x,y\in\Om_T$, and some $\ga<\infty$ depending only on $\al_0$ and $v_0$.
(Note that these estimates imply that $d_0$ generates the same topology as we have already 
on $\Om_T$.)
Moreover, by our restrictions on $T$ from \eqref{Talep}, we can be sure that
$R_0>\frac14-2\ep$, where $R_0$ is defined in Lemma \ref{biholder_dist2}, and so we can set $R=\frac14-2\ep$, in which case \eqref{ball_inclusions}
tells us that
\beq
\label{compact_inclusions_in_m}
B_{g(t)}(x_0,\frac14-2\ep)\subset\subset \m\qquad\text{and}\qquad
B_{d_0}(x_0,\frac14-2\ep)\subset\subset \m,
\eeq
for all $t\in (0,T]$, where $\m$ is the connected component of
$\Interior(\Om_T)$ containing $x_0$.

In summary, we have so far constructed a smooth limit Ricci flow $g(t)$ for $t\in (0,T]$
on $\overline{B_{g_\infty}(x_i,\half-2\ep)}$, and extended its distance function uniformly to $t=0$ on an open subdomain $\m$ that compactly contains both $B_{g(t)}(x_0,\frac14-2\ep)$, for each $t\in (0,T]$, and also $B_{d_0}(x_0,\frac14-2\ep)$.

In what follows, we will need to consider the distance with respect to $g_i(t)$ between image points $\vph_i(x)$ and $\vph_i(y)$, and we pause to verify that such distances are realised by minimising geodesics for arbitrary $x,y$ in the domain of definition of the Ricci flow $g(t)$.
By Part 1 of Lemma \ref{dist_conv_lem},
after omitting finitely many terms in $i$,
the image of the entire flow domain $\overline{B_{g(T)}(x_0,\half-2\ep)}$ under $\vph_i$ must lie within $B_{g_i(T)}(x_0,\half-\frac32 \ep)$, say, which in turn, by the Shrinking Balls Lemma \ref{SBL}, and \eqref{Talep}, must lie within $B_{g_i(t)}(x_0,\half- \ep)$ for all $t\in [0,T]$.
Since the Mollification Theorem told us that 
$B_{g_i(t)}(x_i,1-2\ep)\subset\subset B_{g_i}(x_i,1-\ep)$ where $g_i(t)$ is defined, we see that
for all $x,y\in \overline{B_{g(T)}(x_0,\half-2\ep)}$, the distance between $\vph_i(x)$ and
$\vph_i(y)$ with respect to any $g_i(t)$ is realised by a minimising geodesic lying within
$B_{g_i(t)}(x_i,1-2\ep)$, cf. Remark \ref{dist_rmk}.

Our essential task now is to compare distances $d_0(x,y)$ and $d_{g_i}(\vph_i(x),\vph_i(y))$
for $x$ and $y$ in $\m$.
The rough strategy is as follows. First, by the distance estimates 
\eqref{specialised_main_dist_est2_repeat} 
that came from 
Lemma \ref{biholder_dist2}, we know that $d_0(x,y)$ is close to 
$d_{g(t)}(x,y)$ for $t\in (0,T]$ small. Second, by the convergence of $g_i(t)$ to $g(t)$, we expect that $d_{g(t)}(x,y)$ should be close to $d_{g_i(t)}(\vph_i(x),\vph_i(y))$. 
Third, by the distance estimates \eqref{main_dist_est_for_g_i} coming from the Mollification Theorem, 
$d_{g_i(t)}(\vph_i(x),\vph_i(y))$ should be close to $d_{g_i}(\vph_i(x),\vph_i(y))$. 

{\bf Claim:} As $i\to\infty$, we have  convergence 
$$d_{g_i}(\vph_i(x),\vph_i(y))\to d_0(x,y)$$
uniformly as $x,y$ vary over $\m$. 

\emph{Proof of Claim:}
Let $\de>0$. We must make sure that for sufficiently large $i$, depending on $\de$, 
we have 
\beq
\label{implies_claim}
|d_{g_i}(\vph_i(x),\vph_i(y))- d_0(x,y)|\leq \de\qquad\text{ for every }
x,y\in \m. 
\eeq
By the distance estimates \eqref{specialised_main_dist_est2_repeat} that came from our application of
Lemma \ref{biholder_dist2}, there exists $t_1\in (0,T]$ such that for all $t\in [0,t_1]$,
we have
\beq
\label{est1}
|d_0(x,y)-d_{g(t)}(x,y)|<\de/3\qquad\text{for all }x,y\in \m. 
\eeq
We need a similar estimate for $g_i(t)$. 
By the distance estimates \eqref{main_dist_est_for_g_i} coming from the Mollification Theorem, we know that there exists $t_2\in (0,T]$ such that for all $t\in [0,t_2]$, we have
\beq
\label{est2}
|d_{g_i}(X,Y)-d_{g_i(t)}(X,Y)|<\de/3,\quad\text{whenever }\exists\, s\in [0,T]
\text{ s.t. }
X,Y\in B_{g_i(s)}(x_i,\half-2\ep).
\eeq
We fix $t_0=\min\{t_1,t_2\}$ so that both \eqref{est1} and \eqref{est2} hold for $t=t_0$.
(In fact, we could have naturally picked $t_1$ and $t_2$ the same from the outset.)

By definition of $\Om_T$, and hence $\m$, we have 
\beq
\label{m_inclusion}
\textstyle
\m\subset B_{g(t_0)}(x_0,\frac14-\frac32\ep),
\eeq
and by \eqref{initial_inclusion}
we can pick $r>\frac14-\frac32\ep$ so that
$$\textstyle
B_{g(t_0)}(x_0,2r)\subset\subset B_{g(T)}(x_0,\half-2\ep),$$
where $g(t_0)$ and the maps $\vph_i$ are defined and we have the 
convergence $\vph_i^*g_i(t_0)\to g(t_0)$.
Therefore we can apply Lemma \ref{dist_conv_lem}, with $s=\frac14-\frac32\ep$
and $\hat g$ and $g_i$ there equal to $g(t_0)$ and $g_i(t_0)$ here, respectively, to 
conclude that
\beq
\label{unif_conv}
d_{g_i(t_0)}(\vph_i(x),\vph_i(y)) \to d_{g(t_0)}(x,y)
\eeq
as $i\to\infty$,  uniformly in $x,y\in \m$.

By combining \eqref{est1} and \eqref{est2} for $t=t_0$, and \eqref{unif_conv}, we will have proved 
\eqref{implies_claim} and hence the claim, provided that 
\eqref{est2} holds for $X=\vph_i(x)$ and $Y=\vph_i(y)$. But by \eqref{unif_conv} applied for first $(x,y)=(x,x_0)$ and then $(x,y)=(x_0,y)$, we see by \eqref{m_inclusion} that
$\vph_i(x),\vph_i(y)\in B_{g_i(t_0)}(x_i,\half-2\ep)$ for sufficiently large $i$ as required.
\hfill$//$

What the claim
tells us is that for arbitrarily small $\de>0$, the maps $\vph_i$ are $\de$-Gromov Hausdorff approximations from 
$\m$
to $\vph_i(\m)$ for sufficiently large $i$.
Unusually for such approximations, the maps are smooth.

\bcmt{Note: can't use Lemma \ref{dist_conv_lem} below because it won't work at $t=0$}

Another immediate consequence of  the claim
is that for all $\eta>0$, for sufficiently large $i$ we have $\vph_i(\cb)\subset B_{g_i}(x_i,1/10+\eta)$,
where 
$\cb:=\overline{B_{d_0}(x_0,1/10)}\subset\m$ (recall \eqref{compact_inclusions_in_m})
which is the second inclusion of \eqref{sandwich2}. To obtain the first inclusion, we first need to clarify how large the image $\vph_i(\m)$ is.
By the first part of \eqref{compact_inclusions_in_m}, with $t=T$, we have
$B_{g(T)}(x_0,\frac14-2\ep)\subset\subset \m$.
Therefore, by Part 2 of Lemma \ref{dist_conv_lem} we have
$$\textstyle
\vph_i(\m)\supset \vph_i(B_{g(T)}(x_0,\frac14-2\ep)) \supset B_{g_i(T)}(x_i,
\frac14-3\ep
),$$
after deleting finitely many terms in $i$.
Using once more the Expanding Balls Lemma \ref{EBL2}, and \eqref{Talep}, we can conclude that
$$\textstyle
\vph_i(\m)\supset B_{g_i(T)}(x_i,\frac{1}{4}-3\ep)
\supset
B_{g_i(0)}(x_i,\frac{1}{4}-4\ep)
\supset B_{g_i}(x_i,1/10-\eta).$$
Now that we are sure the image of the maps $\vph_i$ is large enough, we immediately obtain from the Claim
the restricted statement
$$\vph_i(\cb)\supset B_{g_i}(x_i,1/10-\eta),$$
after dropping finitely many terms in $i$, which is the first 
inclusion of \eqref{sandwich2}.

Now we have proved the existence and claimed properties of the maps $\vph_i$, 
it is easy to perturb them to Gromov-Hausdorff approximations 
$(\cb,d_0)\to (\overline{B_{g_i}(x_i,1/10)},d_{g_i})$, and we obtain
the claimed Gromov-Hausdorff convergence
$(\overline{B_{g_i}(x_i,1/10)},d_{g_i})\to (\cb,d_0)$.

Finally we turn to the Lipschitz and H\"older claims of the theorem.
Whichever 
metric $g$ we take on $\m$, the distance $d_g$ will be bi-Lipschitz equivalent to 
$d_{g(T)}$ once we restrict to $\cb$, so we need only prove the claims with $g$ replaced by
the (incomplete) metric $g(T)$
from the Ricci flow we have constructed.
The second inequality of 
\eqref{specialised_main_dist_est2_repeat} tells us that 
$d_{g(T)}(x,y)\leq e^{\al T}d_0(x,y)$, which ensures that the identity map 
$(\cb,d_0)\to (\cb,d_{g(T)})$ is Lipschitz continuous.
Meanwhile, \eqref{holder_est_repeat} tells us that
$d_{g(T)}(x,y)\geq \ga \left[d_0(x,y)\right]^{1+4c_0}$,
which implies that the identity map 
$(\cb,d_{g(T)})\to (\cb,d_0)$ is H\"older continuous with H\"older exponent
$\frac{1}{1+4c_0}$.
\end{proof}

\bcmt{We used to ask that $g$ is complete. This makes it very easy to see the claim above that the distance $d_g$ will be Lipschitz equivalent to $d_{g(T)}$ once we restrict to $\cb$. But it's true anyway, with the caveat that this is the only place in the paper where we consider distance metrics coming from Riemannian metrics where the distance is not necessarily realised by a minimising geodesic.
To see the Lipschitz equivalence, note that because 
$\cb\subset\subset\m$, then in fact we have that for some 
small $\de>0$ and set $\Om\subset\subset\m$, we have that for all
$x\in B_{d_0}(x_0,1/10)$, we have $B_{g}(x_0,\de)\subset\subset U\subset\subset\m$.
Now for points $x,y\in \cb$ with $d_g(x,y)<\de/100$, we must have a minimising geodesic between them with respect to $g$, that lies within 
$U$. By measuring that with respect to $g(T)$, we see that
$d_{g(T)}(x,y)\leq C d_{g}(x,y)$ for $C$ depending on the metrics and $U$. But for $d_g(x,y)\geq \de/100$, this inequality is obvious if we allow $C$ to depend on $\de$ and the supremum of $d_{g(T)}$
over $\cb\times\cb$. Thus we have a Lipschitz bound one way round.
But we can switch the metrics in this argument to get the reverse direction.}

\bcmt{note that we're not saying anything about the regularity of the boundary of $\cb$ within $\m$.}

\section{{Proof of the Anderson, Cheeger, Colding, Tian conjecture}}
\label{ACCTcomplete_sect}

\bcmt{An $n$-dimensional topological manifold is a second countable Hausdorff space that is locally Euclidean of dimension $n$ We're dealing with metric spaces, so second countable is equivalent to separability, and this is clear. I won't mention it}

\begin{proof}[Proof of Corollary \ref{ACCTcomplete}]
First note that it is expected that
it is impossible 
to flow from $(M_i,g_i)$ since we are not assuming uniform global noncollapsing.
As a result, this time it is more convenient to start by appealing to Gromov compactness to get 
a complete limit length space
$$(M_i,d_{g_i},y_i)\to (X,d_X,y_0)$$
in the pointed Gromov-Hausdorff sense, for some subsequence in $i$.

\bcmt{a ref for cptness theorem - BBI 8.1.10}

To show that the topological space $M$ induced by $(X,d_X)$ 
is in fact a  manifold with bi-H\"older charts as claimed in the corollary, we must
show that given an arbitrary $x \in X$, there is a neighbourhood of $x $ that is bi-H\"older homeomorphic to a ball in $\R^3$, or indeed to some open subset of a complete
Riemannian three-manifold.

By definition of the pointed convergence above, and the fact that the limit is a length space, for $r:=d_X(x ,y_0)$ there exists 
a sequence $f_i$ of $\ep(i)$-Gromov-Hausdorff 
approximations $B_{d_X}(y_0,r+1)\to B_{d_{g_i}}(y_i,r+1)$, where $\ep(i)\downto 0$ as $i\to\infty$, with $f_i(y_0)=y_i$ for each $i$.
Defining $x_i:=f_i(x)$, we obtain the pointed Gromov-Hausdorff 
convergence 
$$(\overline{B_{{g_i}}(x_i,1/10)},d_{g_i},x_i)
\to
(\overline{B_{d_X}(x,1/10)},d_X,x),$$
with the limit being compact.

%
%

By the hypotheses \eqref{ACCTcomplete_ests}
and the fact that $d_{g_i}(y_i,x_i)$ converges to $d_X(y_0,x )=r$, and so is bounded, 
we must have a uniform lower bound 
$\VolB_{g_i}(x_i,1)\geq \ti v_0$, for some $\ti v_0>0$ independent of $i$, by Bishop-Gromov.
Therefore we can apply Theorem \ref{preACCT} with 
$v_0$ there equal to $\ti v_0$ here, to show that 
there exist a smooth  three-dimensional manifold without boundary $\m$, containing a point $x_0$,  and a metric
$d_0:\m\times \m\to [0,\infty)$ generating the same topology as $\m$
such that $B_{d_0}(x_0,1/10)\subset\subset\m$ and so that
after passing to a subsequence the compact metric spaces 
$(\overline{B_{g_i}(x_i,1/10)},d_{g_i})$ Gromov-Hausdorff converge to 
$(\cb,d_0)$, where $\cb=\overline{B_{d_0}(x_0,1/10)}$.
In fact, that theorem implies the pointed Gromov-Hausdorff convergence
$$(\overline{B_{g_i}(x_i,1/10)},d_{g_i},x_i)
\to
(\cb,d_0,x_0)
.$$
Moreover, that theorem tells us that the identity map from 
$(\cb,d_0)$ to $(\cb,d_g)$, for any smooth complete metric $g$ on $\m$,
is bi-H\"older.

By uniqueness of compact limits under pointed 
Gromov-Hausdorff convergence \cite[Theorem 8.1.7]{BBI},
$(\cb,d_0)$ must be 
isometric to the restriction of $(X,d_X)$
to the closed  ball centred at $x$ of radius $1/10$, 
via an isometry that identifies $x_0$ and $x$.

\bcmt{In general, for pointed GH convergence, we need boundedly compact and complete to get uniqueness of limits - BBI Thm 8.1.7. Incidentally, `boundedly compact' is inherited from the approximating sequence if we only take complete limits (as we do) - see BBI Ex. 8.1.8. Ditto `length space', of course}

Consequently, we find that
$(B_{d_0}(x_0,1/10),d_0)$ is isometric to $(B_{d_X}(x,1/10),d_X)$, and we have found a bi-H\"older homeomorphism from this neighbourhood
to $(B_{d_0}(x_0,1/10),d_g)$.
\end{proof}

\cmt{
\brmk
\label{ACCT_ext_rmk}
Of other possible results that could be proved with the same technology as in this paper, we mention that Corollary \ref{ACCTcomplete} could be strengthened to conclude that the Gromov-Hausdorff limit is  homeomorphic to a smooth manifold via 
a homeomorphism that is locally bi-H\"older, analogously to Theorem \ref{preACCT}. 
The direct approach to this would be to argue that we can construct a \emph{partial} Ricci flow starting at the limit space, that is a genuine Ricci flow when restricted to balls of finite radius, but whose existence time might converge to zero when restricted to larger and larger balls.
\ermk
}

\section{{Proof of the Global Existence Theorem \ref{global_exist_thm}}}
\label{global_sect}

Before beginning the proof, we recall the following special case of
a result of B.L. Chen \cite[Theorem 3.1]{strong_uniqueness}.
Similar results were independently proved by the first author
\cite[Theorem 1.3]{MilesIMRN}.

\begin{lemma}
\label{Rm_ext_lem}
Suppose $(M^n,g(t))$ is a  Ricci flow for $t\in [0,T]$, not necessarily complete, with the property that for some $y_0\in M$ and $r>0$, and all
$t\in [0,T]$, we have $B_{g(t)}(y_0,r)\subset\subset M$ and 
\begin{equation}
|\Rm|_{g(t)}\leq \frac{c_0}{t} \qquad \text{ throughout } 
B_{g(t)}(y_0,r),\text{ for all } t\in (0,T],
\end{equation}
and for some $c_0\geq 1$.
Then if $|\Rm|_{g(0)}\leq r^{-2}$ on $B_{g(0)}(y_0,r)$, we must have
$$|\Rm|_{g(t)}(y_0)\leq e^{C c_0}r^{-2}$$
where $C$ depends only on $n$.
\end{lemma}

\bcmt{Miles has done an additional check of this precise form of this Chen lemma}

We will also need a slightly nonstandard version of Shi's derivative estimates, as phrased in \cite[Lemma A.4]{JEMS}. For a slightly stronger result, and proof, see \cite[Theorem 14.16]{chowRFTA_V2P2}.

\begin{lemma}
\label{Shi_time0}
Suppose $(M^n,g(t))$ is a  Ricci flow for $t\in [0,T]$, not necessarily complete, with the property that for some $y_0\in M$ and $r>0$,
we have $B_{g(0)}(y_0,r)\subset\subset M$, and $|\Rm|_{g(t)}\leq r^{-2}$ throughout $B_{g(0)}(y_0,r)$ for all $t\in [0,T]$, and 
so that for some $l_0\in \N$ we initially have
$|\grad^l\Rm|_{g(0)}\leq r^{-2-l}$  throughout $B_{g(0)}(y_0,r)$ 
for all $l\in \{1,\ldots, l_0\}$.
Then there exists $C<\infty$ depending only on $l_0$, $n$ and an upper bound for $T/r^2$ such that 
$$|\grad^l \Rm |_{g(t)}(y_0)\leq C r^{-2-l}$$
for every $l\in \{1,\ldots,l_0\}$ and 
all $t\in [0,T]$.
\end{lemma}

\begin{proof}[Proof of Theorem \ref{global_exist_thm}]
Pick any point $x_0\in M$.
For each integer $k\geq 2$, apply Theorem \ref{main_exist_thm2} to the 
manifold $(M,g_0)$, 
with $s_0=k+2$. The conclusion is that there exist $T,v,\al,c_0>0$,
depending only on $\al_0$ and $v_0$, 
so that 
there exists a Ricci flow
$g_k(t)$ defined on $B_{g_0}(x_0,k)$ for $t\in [0,T]$ with $g_k(0)=g_0$ on $B_{g_0}(x_0,k)$ with the properties that
\begin{equation}
\label{g_k_facts}
\left\{
\begin{aligned}
& \Ric_{g_k(t)}\geq -\al\\ 
& |\Rm|_{g_k(t)}\leq c_0/t 
\end{aligned}
\right.
\end{equation}
on $B_{g_0}(x_0,k)$ for all $t\in (0,T]$.

For $r_0>0$ and $k\geq r_0+2$, Lemma \ref{Rm_ext_lem} can now be applied to $g_k(t)$ centred at arbitrary points
$y_0\in B_{g_0}(x_0,r_0+1)$.
Note that the curvature estimate $|\Rm|_{g(t)}\leq c_0/t$
holds on $B_{g_0}(y_0,1)$, and by the Shrinking Balls Lemma \ref{SBL}, this contains $B_{g_k(t)}(y_0,1/2)$ for all $t\in [0,T]$ after possibly reducing $T$ to a smaller positive value depending only on $c_0$, which in turn depends only on $\al_0$ and $v_0$.
The output of Lemma \ref{Rm_ext_lem} is that 
there exists $K_0<\infty$, depending only on 
$\sup_{B_{g_0}(x_0,r_0+2)}|\Rm|_{g_0}$ and $c_0$
and in particular
independent of $k$, such that 
$|\Rm|_{g_k(t)}\leq K_0$ 
throughout $B_{g_0}(x_0,r_0+1)$, for all $t\in [0,T]$.

We may then apply Lemma \ref{Shi_time0} to $g_k(t)$ 
centred at arbitrary points $y_0\in B_{g_0}(x_0,r_0)$.
This time we deduce that for each $l\in \N$
there exists $K_1<\infty$, depending only on 
$l$, $g_0$ and $r_0$, and in particular
independent of $k$, such that 
$$|\grad^l \Rm|_{g(t)}\leq K_1$$ 
throughout $B_{g_0}(x_0,r_0)$, for all $t\in [0,T]$.
%
%
Working in coordinate charts, by Ascoli-Arzel\`a we can pass to a subsequence in $k$
and obtain a smooth limit Ricci flow $g(t)$ on $B_{g_0}(x_0,r_0)$ 
for $t\in [0,T]$ with $g(0)=g_0$
(it is not necessary to take a Cheeger-Hamilton limit here since $g_k(0)=g_0$ on $B_{g_0}(x_0,r_0)$ for sufficiently large $k$).
Moreover, the limit will inherit the curvature bounds
\begin{equation}
\label{gt_facts}
\left\{
\begin{aligned}
& \Ric_{g(t)}\geq -\al\\
& |\Rm|_{g(t)}\leq c_0/t
\end{aligned}
\right.
\end{equation}
on $B_{g_0}(x_0,r_0)$ for all $t\in (0,T]$.

We can now repeat this process for larger and larger $r_0\to\infty$, and take a diagonal subsequence to obtain a smooth limit Ricci flow $g(t)$ on the whole of $M$ for $t\in [0,T]$ with $g(0)=g_0$.
This limit flow must be complete because for arbitrary $t\in (0,T]$ 
and arbitrarily large $r>0$, the Shrinking Balls Lemma \ref{SBL} tells us that $B_{g(t)}(x_0,r)\subset B_{g_0}(x_0,r+\be\sqrt{c_0 t})\subset\subset M$.

Given our curvature control \eqref{gt_facts} and the completeness of
$g(t)$, 
the distance estimates are standard, although they also follow from our more elaborate Lemma \ref{biholder_dist2}.
\end{proof}

\section{Starting a Ricci flow with a Ricci limit space}
\label{RF_from_RLS_sect}

In this section we prove Theorem \ref{RFfromRLS}.

\begin{proof}
Because we are making a stronger assumption in Theorem \ref{RFfromRLS}
than in Corollary \ref{ACCTcomplete}, we can apply Theorem 
\ref{global_exist_thm} to each $(M_i,g_i)$ to give a sequence of 
Ricci flows $g_i(t)$ on $M_i$ with $g_i(0)=g_i$, defined over a uniform time interval
$[0,T]$, and enjoying the uniform estimates
\begin{equation}
\label{i_global_conc_ests}
\left\{
\begin{aligned}
& \Ric_{g_i(t)}\geq -\al & &\\
& \VolB_{g_i(t)}(x,1)\geq v>0 \qquad & &\text{for all }x\in M_i\\
& |\Rm|_{g_i(t)}\leq c_0/t \qquad & &\text{throughout } M_i
\end{aligned}
\right.
\end{equation}
for all $t\in (0,T]$,
and
\beq
\label{i_global_dist_est}
d_{g_i(t_1)}(x,y)-\be\sqrt{c_0}(\sqrt{t_2}-\sqrt{t_1})
\leq d_{g_i(t_2)}(x,y)
\leq e^{\al(t_2-t_1)}d_{g_i(t_1)}(x,y),
\eeq
for any $0\leq t_1\leq t_2\leq T$, and any $x,y\in M_i$.
Because of the uniform curvature bounds for positive times, Hamilton's compactness theorem tells us that we can pass to a subsequence in $i$ so that $(M_i,g_i(t),x_i)\to (M,g(t),x_\infty)$ in the smooth Cheeger-Gromov sense, where $M$ is a smooth manifold, $g(t)$ is a complete Ricci flow on $M$ for $t\in (0,T]$,
$x_\infty\in M$ and we pass the estimates \eqref{i_global_conc_ests} and \eqref{i_global_dist_est}
to the limit to give
\begin{equation}
\label{duplicate_global_conc_ests}
\left\{
\begin{aligned}
& \Ric_{g(t)}\geq -\al & &\\
& \VolB_{g(t)}(x,1)\geq v>0 \qquad & &\text{for all }x\in M\\
& |\Rm|_{g(t)}\leq c_0/t \qquad & &\text{throughout } M
\end{aligned}
\right.
\end{equation}
for all $t\in (0,T]$, and 
\beq
\label{dup_global_dist_est}
d_{g(t_1)}(x,y)-\be\sqrt{c_0}(\sqrt{t_2}-\sqrt{t_1})
\leq d_{g(t_2)}(x,y)
\leq e^{\al(t_2-t_1)}d_{g(t_1)}(x,y),
\eeq
for any $0< t_1\leq t_2\leq T$, and any $x,y\in M$.
This final estimate tells us that there exists a metric $d_0$ on $M$ to which $d_{g(t)}$ converges locally uniformly as $t\downto 0$.
It also tells us, when combined with the smooth convergence of $(M_i,g_i(t),x_i)$ to $(M,g(t),x_\infty)$ for positive times and \eqref{i_global_dist_est} that 
$$(M_i,d_{g_i},x_i)\to (M,d_0,x_\infty)$$
in the pointed Gromov-Hausdorff sense, which can be considered an easier version of the argument in Section \ref{ricci_lims_sect}.
\end{proof}

{\sc MS: institut f\"ur analysis und numerik (IAN),
universit\"at magdeburg, universit\"atsplatz 2,
39106 magdeburg, germany}

{\sc PT: mathematics institute, university of warwick, coventry, CV4 7AL,
uk}\\
\url{https://www.warwick.ac.uk/~maseq}
\end{document}